\documentclass[a4paper,11pt,draft]{amsart}
\usepackage[margin=22mm]{geometry}
\usepackage{amssymb, amsmath, amsthm, amscd}

\usepackage{amsmath}
\allowdisplaybreaks[4]

\usepackage[T1]{fontenc}
\usepackage{eucal,mathrsfs,dsfont}
\usepackage{color}
\usepackage{mathrsfs}

\definecolor{mno}{rgb}{0.5,0.1,0.5}

\newcommand{\R}{\mathbb R}

\newcommand{\Pp}{\mathbb P}
\newcommand{\Ee}{\mathbb E}

\newcommand{\I}{\mathbf 1}

\def\<{\langle}
\def\>{\rangle}

\newtheorem{theorem}{Theorem}[section]
\newtheorem{lemma}[theorem]{Lemma}
\newtheorem{proposition}[theorem]{Proposition}
\newtheorem{corollary}[theorem]{Corollary}

\theoremstyle{definition}

\newtheorem{example}[theorem]{Example}

\begin{document}
\allowdisplaybreaks
\title[Two-sided heat kernel estimates for Schr\"{o}dinger operators
]
{\bfseries Two-sided heat kernel estimates for Schr\"{o}dinger operators
with
unbounded potentials
}

\author{Xin Chen\qquad Jian Wang}

\date{}

 \thanks{\emph{X.\ Chen:}
   Department of Mathematics, Shanghai Jiao Tong University, 200240 Shanghai, P.R. China. \texttt{chenxin217@sjtu.edu.cn}}
 
\thanks{\emph{J.\ Wang:}
School  of Mathematics and Statistics \&  Fujian Key Laboratory of Mathematical
Analysis and Applications (FJKLMAA) \&  Center for Applied Mathematics of Fujian Province (FJNU), Fujian Normal University, 350007 Fuzhou, P.R. China. \texttt{jianwang@fjnu.edu.cn}}

\maketitle

\begin{abstract}Consider the Schr\"odinger operator
$
\mathcal L^V=-\Delta+V
$ on $\R^d$,
where
$V:\R^d\to [0,\infty)$ is a nonnegative and locally bounded potential on $\R^d$ so that for all $x\in \R^d$ with $|x|\ge 1$,
$c_1g(|x|)\le V(x)\le c_2g(|x|)$ with some constants $c_1,c_2>0$ and a nondecreasing and strictly positive function $g:[0,\infty)\to [1,+\infty)$ that satisfies
$g(2r)\le c_0 g(r)$ for all $r>0$ and $\lim_{r\to \infty} g(r)=\infty.$ We establish global in time and qualitatively sharp bounds for the heat kernel of the associated Schr\"{o}dinger semigroup
by the probabilistic method.
In particular, we can present global in space and time two-sided bounds of heat kernel even when the  Schr\"{o}dinger semigroup is not intrinsically ultracontractive. Furthermore, two-sided estimates for the corresponding Green's functions are also obtained.

\medskip


\medskip

\end{abstract}
\allowdisplaybreaks

\section{Introduction and main result}
The study of heat kernel bounds for differential operators on domains of $\R^d$ or Riemannian manifolds has a long history.
We refer the
reader
for an account on this to the monographs
of Davies \cite{Da}, Grigor'yan \cite{Gr} and Ouhabaz \cite{Ou}.  A fundamental result in this direction proved in \cite{A} states that the fundamental solution (that is, the heat kernel)  $p(t,x,y)$ of a second
order uniformly parabolic equation in divergence form on $\R^d$ enjoys the following Gaussian estimates
$$p(t,x,y)\asymp t^{-d/2}\exp\left(-\frac{|x-y|^2}{t}\right),\quad t>0, x,y\in \R^d.$$
\emph{Here and in what follows, for two functions $f$ and $g$, $f\asymp g$ means
that there are positive constants $c_1,c_2,c_3,c_4$ such that $c_1f(c_2r)\le g(r)\le c_3 f(c_4r)$.}

Since then many efforts are devoted to characterizing functions $V$ for which the heat kernel
of the Schr\"{o}dinger operator $\mathcal L:=-\Delta+V$ (or the second-order differential elliptic
operator with the Schr\"{o}dinger perturbation $V$) is comparable with the Gaussian estimates. For example, this happens when $V\in L^p(\R^d)$ with $p>d/2$ or the potential $V$ belongs to the so-called Kato class;
see Bogdan, Dziuba\'{n}ski and Szczypkowski \cite{BDS}
for the recent study on this topic.
However, in general such Gaussian estimates can not hold. A typical and important
example is the Schr\"{o}dinger operator with inverse square potential, that is,
$$\mathcal L_\lambda:= -\Delta+\frac{\lambda}{|x|^2},$$ where $-(d-2)^2/4\le \lambda<0$; see
Ishige, Kabeya and Ouhabaz \cite{IKO}, Milman and Semenov \cite{Mis} and Zhang \cite{Z1,Z2} for explicit non-Gaussian estimates for the heat kernel of
$\mathcal L_\lambda$.

Since the Schr\"{o}dinger operator and its properties play an important role in geometry, analysis, probability, quantum mechanics and other subjects,
the following important question naturally arises (e.g., see \cite[p.\ 371]{Z1} and \cite[p.\ 344]{Z2}) ---

\emph{Does there exist a global estimate on the heat kernel of $\mathcal L^V:=-\Delta+V$, which
reveals an explicit dependence on the potential $V$?}

It is obvious to see that the heat kernel
$p(t,x,y)$ of $\mathcal L^V$ on $\R^d$
satisfies
$$\frac{c_d e^{-\|V\|_\infty t} }{t^{d/2}}\exp\left(-\frac{|x-y|^2} {4t}\right)\le  p(t,x,y)\le \frac{c_d e^{\|V\|_\infty t} }{t^{d/2}}\exp\left(-\frac{|x-y|^2} {4t}\right).$$
The bound above is
not qualitatively sharp in time
since the functions $e^{-\|V\|_\infty t}$ and $e^{\|V\|_\infty t}$ appear in both sides simultaneously. In particular, the bound above is not useful when the potential $V$ is unbounded on $\R^d$. One can imagine in this case that the bounds of $p(t,x,y)$ should depend on the potential $V$ in a very
implicit manner, where the Gaussian estimates do not hold.
A number of authors have studied heat kernel estimates corresponding to Schr\"{o}dinger operators
with unbounded potentials, see e.g. \cite{Ku,MPR,MS1,S} and references
 therein.
In particular, when the potential
$V$ blows up near infinity, precise on-diagonal estimates of the heat kernel $p(t,x,y)$
were obtained by Sikora \cite{S}, also see \cite{Ku,MPR,MS1} for more details.
Indeed, there are a few deeper results in this special setting (see \cite{BJP} and the references therein). For example, suppose that $V(x)=|x|^\alpha$ with $\alpha>2$, Davies and
Simon proved in \cite[Theorem 6.3]{DS84} that the Schr\"{o}dinger
semigroup associated with the operator $\mathcal L^V=\Delta+ V$ satisfies the so-called
intrinsically ultracontractive property. In particular, according to \cite[Theorem 6.3]{DS84} and \cite[Theorem 4.2.5 and Corollary 4.5.8]{Da}, we know that for any $\varepsilon>0$,
there exists $T>0$ so that for all $t>T$ and $x,y\in \R^d$,
$$(1-\varepsilon)e^{-\lambda_1t}\varphi_1(x)\varphi_1(y)\le p(t,x,y)\le (1+\varepsilon)e^{-\lambda_1t}\varphi_1(x)\varphi_1(y),
$$
where $\varphi_1(x)$ is a ground state (eigenfunction corresponding to the smallest eigenvalue $\lambda_1$) of
the operator $\mathcal L^V$ (also see \cite[Theorem 7]{S} for the explicit expression  of $\varphi_1$).
Similar
bounds have been also obtained in Metafune, Pallara and Rhandi \cite{MPR}, where estimates of the form $p(t,x,y)\le c(t)\varphi(x)\varphi(y)$
were derived with more accurate estimates of $c(t)$ as $t$ tends to $0$.

However, to the best of our knowledge, global qualitatively sharp two-sided bounds (especially the off-diagonal estimates) are still unknown
for general potentials $V:\R^d\to [0,\infty)$ which tend to  
$\infty$ as $|x|\to \infty$.
The main purpose of this paper is to address this problem affirmatively.

\ \

Throughout the paper, let $\mathcal L^V$ be the Schr\"odinger operator on $\R^d$
\begin{equation}\label{e1-1}
\mathcal L^V=-\frac{1}{2}\Delta+V,
\end{equation}
where $\Delta:=\sum_{i=1}^d \frac{\partial^2}{\partial x_i^2}$ denotes the Laplacian operator on $\R^d$, and
$V:\R^d\to [0,\infty)$ is a nonnegative and locally bounded potential on $\R^d$.
It is well known that (see e.g. \cite[Chapter 3]{CZ}) there exists a
Schr\"odinger semigroup $\{T_t^V\}_{t\ge 0}$ associated with the operator $\mathcal L^V$ defined by \eqref{e1-1}, and, by the Feynman-Kac formua,
$\{T_t^V\}_{t\ge 0}$ enjoys the following probabilistic representation
\begin{equation}\label{e1-4}
T_t^V f(x)=\Ee_x\left[f(B_t)\exp\left(-\int_0^t V(B_s)\,ds\right)\right],\quad f\in C_b(\R^d),
\end{equation}
where $\{B_t\}_{t\ge 0}$ is the standard $\R^d$-valued Brownian motion,
and
$\Pp_x$ and $\Ee_x$ denote
the probability and the expectation corresponding to $\{B_t\}_{t\ge 0}$ with initial value $x\in \R^d$ respectively.
Furthermore, according to \cite[Chapter 3]{CZ} (see also \cite[Theorem B.7.1]{Sim}), there exists a jointly continuous density $p:\R_+\times \R^d\times \R^d \to \R_+$ with respect to the Lebesgue measure such that
\begin{equation}\label{e1-5}
T_t^V f(x)=\int_{\R^d}p(t,x,y)f(y)\,dy,\quad  f\in C_b(\R^d),\ t>0,\ x\in \R^d.
\end{equation}  
Moreover, since $\mathcal L^V$ is symmetric on $L^2(\R^d;dx)$, $p(t,x,y)=p(t,y,x)$ for all $t>0$ and $x,y\in \R^d$.  
In the following, let $q(t,x,y)$ be the transition density function (i.e., heat kernel) of the standard $d$-dimensional Brownian motion $\{B_t\}_{t\ge0}$; that is,
\begin{equation}\label{e1-6}
q(t,x,y)=(2\pi t)^{-d/2}\exp\left(-\frac{|x-y|^2}{2t}\right),\quad t>0,\ x,y\in \R^d.
\end{equation}
Since $V\ge0$, it is obvious that $$p(t,x,y)\le q(t,x,y),\quad t>0, x,y\in \R^d.$$

\ \

To state our main result, we need some assumption
and notations.
Let $\mathscr{G}$ be the class of nondecreasing and strictly positive functions $g:[0,\infty)\to [1,+\infty)$ such that
$$\lim_{r\to \infty}g(r)=\infty$$
and
there exists a constant $c_0>0$ so that for all $r\ge 0,$
\begin{equation}\label{e1-2a}
g(2r)\le c_0 g(r).
\end{equation}
In this paper, we always make the following assumption on $V$:

\noindent  {\bf Assumption (H)}: {\it There exist positive constants $C_1,C_2$ and $g\in \mathscr{G}$ so that
\begin{equation}\label{e1-3}
\begin{split}
  C_1g(|x|)\le V(x)\le C_2g(|x|),&\quad  x\in \R^d\ {\rm with}\ |x|> 1.
\end{split}
\end{equation}}

For any $s\ge0$, set
$$t_0(s):=\frac{1+s}{\sqrt{g(s)}},\quad s>0.$$ We say $t_0:[0,\infty)\to \R_+$ is almost increasing (resp.\ almost decreasing), if
there exists an increasing (resp.\ a decreasing) function $h:[0,+\infty) \to [1,+\infty)$ such that for some
positive constants $C_*,C^*$,
$$
C_*h(s)\le t_0(s)\le C^*h(s), \quad s\ge0.$$ For any $a,b>0$, $a\wedge b=\min\{a,b\}$ and $a\vee b=\max\{a,b\}$.

\begin{theorem}\label{thm1} Under {\bf Assumption (H)},  it holds for any $C_0>0$ that
\begin{itemize}
\item[(i)] for all $x,y\in \R^d$ and $0<t\le C_0t_0(|x|\wedge |y|)$,
\begin{equation}\label{e:thm1-1}p(t,x,y)\asymp t^{-d/2}e^{-\frac{|x-y|^2}{t}}\,
\exp\left(-\left(t \min\{g(|x|),g(|y|)\}+|x-y|\sqrt{\max\{g(|x|), g(|y|)\}}\right)\right).\end{equation}

\item[(ii)] for all $x,y\in \R^d$ and $t\ge C_0 t_0(|x|\wedge |y|)$,
\begin{equation}\label{e:thm1-2}
p(t,x,y)\asymp e^{-t}\psi(t,x)\psi(t,y),
\end{equation}
where
\begin{equation*}
\psi(t,x):=\exp\left(-\left((1+|x|)\sqrt{g(|x|)}+\frac{(1+|x|)^2}{t}\right)\right),\quad t>0,\ x\in \R^d.
\end{equation*}
\end{itemize}
\end{theorem}

We make some comments on Theorem \ref{thm1} and its proof.

\begin{itemize}

\item [(i)] Different from the approaches adopted in \cite{Ku,MPR,MS1,S}, we mainly use the probabilistic method to prove Theorem \ref{thm1}.
Two key tools used in our proof are the probabilistic representation \eqref{l2-4-2} for the heat kernel $p(t,x,y)$ and an explicit distribution of
$(\tau_{B(x,r)},B_{\tau_{B(x,r)}})$ with $\tau_{B(x,r)}$ being the first exit time of the Brownian motion $\{B_t\}_{t\ge 0}$ from the
ball
$B(x,r)$, which gives us a new way
to obtain two-sided estimates of $p(t,x,y)$ by carefully studying
the interaction between the potential $V$ and the behavior of Brownian motion $\{B_t\}_{t\ge 0}$ (i.e., the behavior of
$\{B_t\}_{t\ge 0}$ to visit the positions where $V$ takes different values), see the proofs of
Lemmas \ref{l2-4}, \ref{l2-2}, \ref{l3-2} and \ref{l4-2} below for more details.

\item[(ii)] In this paper we only consider
the heat kernel for the Schr\"odinger operator $\mathcal L^V=-\frac{1}{2}\Delta+V$, but the approach works for more general framework. Indeed, according to our proof, with explicit estimates for
the derivative of Dirichlet heat kernel (see \eqref{l2-5-1} and \eqref{l2-1-1}) at hand, we can
consider more general cases where the Laplacian operator $\Delta$ is replaced by second order differential operators with divergence
form in $\R^d$, or the Laplacian-Beltrami operator on a Riemannian manifold. Furthermore, by applying the probabilistic
representation \eqref{l2-4-2} of $p(t,x,y)$, we can also obtain two-sided heat kernel estimates for $\mathcal L^V=-\Delta+V$ with $V$ belonging to a class of potential functions
decaying to zero at infinity in the separate paper \cite{CW1}.

\item [(iii)] For simplicity of the notation, in this paper we make the {\bf Assumption (H)} on the
potential function $V$. After carefully tracking
the proof we know immediately that the upper bounds in \eqref{e:thm1-1} and \eqref{e:thm1-2} still hold when only the lower bound of $V$ in {\bf Assumption (H)} is posed, while
the corresponding lower bounds of $p(t,x,y)$ only depend the upper bound of $V$. Moreover, the condition \eqref{e1-2a} is still not essential for our result. For example,
consider the potential $V(x)$ so that
$C_1g(C_2 |x|)\le V(x)\le C_3 g(C_4|x|)$ for all $|x|\ge 1$ with $g(r)=e^r$.
Then, in the proof of Theorem \ref{thm1}, the property that $c_1g(|x|)\le \sup_{z\in B(x,|x|/3)}V(z)\le c_2g(|x|)$ will be replaced by
$c_3g(c_4|x|)\le \sup_{z\in B(x,|x|/3)}V(z)\le c_5g(c_6|x|)$. Hence, by our arguments we can see that \eqref{e:thm1-1} and \eqref{e:thm1-2} are still true.
\end{itemize}

As mentioned before, compared with the known results in \cite{Ku,MPR,MS1,S}, Theorem \ref{thm1} shows qualitatively sharp
two-sided estimates of $p(t,x,y)$  globally in time and space. To the best our knowledge, Theorem \ref{thm1} covers on-diagonal estimates in the literature and, in particular, presents (two-sided) off-diagonal estimates in
the first time. Below, we consider a concrete example that $V(x)=|x|^\alpha$ for $\alpha>0$. Note that, when $\alpha\in (0,2]$, the
associated Schr\"{o}dinger
semigroup $\{T_t^V\}_{t\ge 0}$ is not intrinsically ultracontractive; however, Example \ref{ex1-1} below even presents two-sided estimates of $p(t,x,y)$ for full time, which also seem to be new.

\begin{example}\label{ex1-1}\it
Assume that $V(x)=|x|^\alpha$ for $\alpha>0$.
Then for every $x,y\in \R^d$ with $|x|\le |y|$ and $t>0$,
\begin{align*}
p(t,x,y)\asymp
\begin{cases}
t^{-d/2}e^{-\frac{|x-y|^2}{t}}e^{-t(1+|x|)^\alpha-|x-y|(1+|y|)^{{\alpha}/{2}}},
&\quad t\le (1+|x|)^{1-{\alpha}/{2}},\\
e^{-t}e^{ -(1+|y|)^{1+{\alpha}/{2}} -\frac{(1+|y|)^2}{t}},&\quad t>
(1+|x|)^{1-{\alpha}/{2}}.
\end{cases}
\end{align*}
In particular,
\begin{itemize}

\item [(i)]for every $x,y\in \R^d$  
with $|x|\le |y|$ satisfying
$|x-y|\le |y|/2$
or $|y|\le 2$, it holds that
\begin{equation}\label{ex1-1-2}
\begin{split}
p(t,x,y)\asymp
\begin{cases}
t^{-d/2}e^{-\frac{|x-y|^2}{t}}e^{-t(1+|y|)^{\alpha}},
\ & t\le \max\left\{(1+|x|)^{1-{\alpha}/{2}}, (1+|y|)^{1-{\alpha}/{2}}\right\},\\
e^{-t}e^{-(1+|y|)^{1+{\alpha}/{2}}},\ &t> \max\left\{(1+|x|)^{1-{\alpha}/{2}}, (1+|y|)^{1-{\alpha}/{2}}\right\}.
\end{cases}
\end{split}
\end{equation}

\item [(ii)] for every $x,y\in \R^d$ with
$|x|\le |y|$ satisfying
$|x-y|>|y|/2$
and $|y|>2$, it holds that
\begin{equation}\label{ex1-1-3}
\begin{split}
p(t,x,y)\asymp
\begin{cases}
t^{-d/2}e^{-\frac{|x-y|^2}{t}}e^{-(1+|y|)^{1+{\alpha}/{2}}},
\ & t\le\max\left\{(1+|x|)^{1-{\alpha}/{2}}, (1+|y|)^{1-{\alpha}/{2}}\right\},\\
e^{-t}e^{-(1+|y|)^{1+{\alpha}/{2}}},\ &t>\max\left\{(1+|x|)^{1-{\alpha}/{2}}, (1+|y|)^{1-{\alpha}/{2}}\right\}.
\end{cases}
\end{split}
\end{equation}
\end{itemize}
\end{example}

In Example \ref{ex1-1} above, we write the estimates of $p(t,x,y)$ in
a more explicit way which reveals various effects of the potential $V(x)=|x|^\alpha$ at
different times $t>0$ and different positions $x,y\in \R^d$.
For large time scale
$t>\max\{(1+|x|)^{1-{\alpha}/{2}}, (1+|y|)^{1-{\alpha}/{2}}\}$ (which depends on the positions $x,y\in \R^d$
and the potential $V$), the Gaussian type bound $t^{-d/2}e^{-\frac{|x-y|^2}{t}}$ has been
totally covered by the estimate for the ground state $\phi(x)\asymp e^{-(1+|x|)^{1+{\alpha}/{2}}}$ (see \cite{Da,DS84}). For small time scale $t\le \max\{(1+|x|)^{1-{\alpha}/{2}}, (1+|y|)^{1-{\alpha}/{2}}\}$, the dominant bound
is the Gaussian type bound multiplying some extra term. In particular, when
$x$ is close to $y$
(i.e., $|x-y|\le |y|/2$ or $|y|\le 2$),
such term is $e^{-t(1+\max\{|x|,|y|\})^\alpha}$, while this term becomes $e^{-(1+\max\{|x|,|y|\})^{1+{\alpha}/{2}}}$
when the distance between $x$ and $y$ is relatively large.
In particular, $p(t,x,y)$ enjoys two-sided Gaussian type
bounds
as long as $t\le \min\{(1+|x|)^{-\alpha}, (1+|y|)^{-\alpha}\}$.

\ \

As an application of Theorem \ref{thm1}, we have the following two-sided estimates for Green's function  of
the Schr\"odinger operator $\mathcal L^V=-\frac{1}{2}\Delta+V$.
\begin{proposition}\label{t1-3}
Suppose {\bf Assumption (H)} holds. Let $G(x,y):=\int_0^\infty p(t,x,y)\,dt$ be Green's function
associated with the Schr\"odinger semigoup $\{T_t^V\}_{t\ge 0}$. Then, for all $x,y\in \R^d$,
\begin{equation}\label{t1-3-1}
G(x,y)\asymp |x-y|^{-(d-2)}\cdot \Gamma(x,y),
\end{equation}
where
\begin{equation}\label{t1-3-1a}
\begin{split}
\Gamma(x,y)& 
= \exp\left(-|x-y|\sqrt{\max\{g(|x|),g(|y|)\}}\right)
\cdot
\begin{cases}
1,\ &d\ge3,\\
\max\left\{\log\left(\frac{1}{|x-y|\sqrt{\max\{g(|x|),g(|y|)\}}}\right),1\right\},\ &d=2,\\
\frac{1}{|x-y|\sqrt{\max\{g(|x|),g(|y|)\}}},\ &d=1.
\end{cases}
\end{split}
\end{equation}
\end{proposition}

The
estimates
for Green's function $G(x,y)$
when $|x-y|$ tends to infinity
were established by Carmona and Simon \cite{CS} and Schroeder \cite{Sh}, but Proposition \ref{t1-3} seems to be the first time to present the estimate \eqref{t1-3-1} of $G(x,y)$ for full space region. Note that
$$\max\left\{\log\left(\frac{1}{|x-y|\sqrt{\max\{g(|x|),g(|y|)\}}}\right),1\right\}\asymp \log\left(\frac{1}{|x-y|}\right)$$ when $|x-y|\le \frac{1}{\sqrt{\max\{g(|x|),g(|y|)\}}}$. When $d\ge 2$, according to
\eqref{t1-3-1},
\begin{align*}
G(x,y)\asymp
\begin{cases}
Q(x,y),\ \ &|x-y|\le \frac{1}{\sqrt{\max\{g(|x|),g(|y|)\}}},\\
|x-y|^{-(d-2)}\exp\left(-|x-y|\sqrt{\max\{g(|x|),g(|y|)\}}\right),\ &|x-y|> \frac{1}{\sqrt{\max\{g(|x|),g(|y|)\}}},
\end{cases}
\end{align*}
where $Q(x,y)$ denotes Green's function of the Laplacian operator $\Delta$ on $\R^d$;
when $d=1$, it holds that
\begin{align*}
G(x,y)\asymp
\begin{cases}
\frac 1{\sqrt{\max\{g(|x|),g(|y|)\}}},\ \ &|x-y|\le \frac{1}{\sqrt{\max\{g(|x|),g(|y|)\}}},\\
|x-y|^{-(d-2)}\exp\left(-|x-y|\sqrt{\max\{g(|x|),g(|y|)\}}\right),\ &|x-y|> \frac{1}{\sqrt{\max\{g(|x|),g(|y|)\}}}.
\end{cases}
\end{align*}
The estimates above illustrate that when $d\ge 2$, $G(x,y)\asymp Q(x,y)$ in the small distance scale $|x-y|\le \frac{1}{\sqrt{\max\{g(|x|),g(|y|)\}}}$ (which depends
on the positions $x,y\in \R^d$
and the potential $V$), which indicates that the potential
$V$ only
has a remarkable effect in the large distance scale;
when $d=1$, the potential
$V$ makes a dominant
role in the estimate of $G(x,y)$ even in the small distance scale.

\ \

The rest of the paper is arranged as follows. In the next section, we will present some preliminary estimates. Then, Sections \ref{section3} and \ref{section4} are devoted to the proof of Theorem \ref{thm1} when $0<t\le C_0t_0(|x|\wedge |y|)$ and $t\ge C_0 t_0(|x|\wedge |y|)$, respectively. In Section \ref{sec5} we will give the proofs
of Example \ref{ex1-1} and Proposition \ref{t1-3}.

\section{Preliminaries}

For any $D\subset \R^d$, let $\tau_D=\inf\{t\ge 0: B_t\notin  D\}.$
We begin with the following stochastic representation for $p(t,x,y)$.

\begin{lemma}\label{l2-6}
Let $U$ be a domain of $\R^d$. Then, for every $x\in U$ and $y\notin \bar U$,
\begin{equation}\label{l2-4-2}
p(t,x,y)=\Ee_x\left[\exp\left(-\int_0^{\tau_U}V(B_s)\,ds\right)\I_{\{\tau_U\le t\}}p\left(t-\tau_U,B_{\tau_U},y\right)\right].
\end{equation}
\end{lemma}
\begin{proof}
Since $y\notin \bar U$,  there exists an open set $D$ such that
$y\in D \subset U^c$.
Taking any nonnegative function
$f\in C_b(\R^d)$ with ${\rm supp}[ f]\subset D$, it holds that
\begin{align*}
 T_t^Vf(x)
&=\Ee_x\left[f(B_t)\exp\left(-\int_0^t V(B_s)\,ds\right)\right]\\
&=\Ee_x\left[f(B_t)\exp\left(-\int_0^{\tau_U}V(B_s)\,ds-\int_{\tau_U}^tV(B_s)\,ds\right)\I_{\{\tau_U\le t\}}\right]\\
&=\Ee_x\left\{\exp\left(-\int_0^{\tau_U}V(B_s)\,ds\right)\I_{\{\tau_U\le t\}}\Ee_{B_{\tau_U}}\left[f(B_{t-\tau_U})\exp\left(-\int_0^{t-\tau_U}V(B_s)\,ds\right)\right]\right\}\\
&=\Ee_x\left[\exp\left(-\int_0^{\tau_U}V(B_s)\,ds\right)\I_{\{\tau_U\le t\}}\left(\int_{\R^d}p\left(t-\tau_U,B_{\tau_U},z\right)f(z)\,dz\right)\right]\\
&=\int_{\R^d}\Ee_x\left[\exp\left(-\int_0^{\tau_U}V(B_s)\,ds\right)\I_{\{\tau_U\le t\}}p\left(t-\tau_U,B_{\tau_U},z\right)\right]f(z)\,dz,
\end{align*}
where the second equality follows from the fact that $f(B_t)=0$ when $\tau_U>t$ thanks to ${\rm supp}[ f] \cap U=\emptyset$, in the third equality
we used the strong Markov property, and the fourth equality is a consequence of the Feynman-Kac formula \eqref{e1-4}.

Since $f$ is chosen arbitrarily,  the above equality implies the desired assertion \eqref{l2-4-2} immediately.
\end{proof}

For any $D\subset \R^d$, let $q_D(t,x,y)$ be the Dirichlet heat kernel of the standard Brownian motion $\{B_t\}_{t\ge 0}$ on $D$.
Denote by $B(x,r):=\{z\in \R^d:|z-x|<r\}$ the Euclidean ball with center $x\in \R^d$ and radius $r>0$.
The following characterization for  the distribution of $(\tau_D,B_{\tau_D})$ was proved in \cite[Theorem 1.1]{H}.
\begin{lemma}
For every $x\in \R^d$ and $r>0$ we have
\begin{equation}\label{l2-5-1}
\Pp_x(\tau_{B(x,r)}\in dt, B_{\tau_{B(x,r)}}\in dz)=\frac{1}{2}\frac{\partial q_{B(x,r)}(t,x,\cdot)}{\partial n}(z)\,\sigma(dz)\,dt,
\end{equation}
where $\sigma(dz)$ denotes the Lebesgue surface measure on $\partial B(x,r)$ $($in particular, when $d=1$, $\sigma(dz)$ is the Dirac measure
at the boundary$)$,
and  $\frac{\partial q_{B(x,r)}(t,x,\cdot)}{\partial n}(z)$ denotes the exterior normal derivative
of $q_{B(x,r)}(t,x,\cdot)$ at the point $z\in D$.
\end{lemma}

\begin{lemma}\label{l2-1}
There exist positive constants $C_i$, $i=1,\cdots,4$, so  that for any $R>0$, $x\in \R^d$, $t>0$ and $y\in \partial B(x,R)$,
\begin{equation}\label{l2-1-1}
\begin{split}
\frac{C_1R}{t^{d/2+1}}\exp\left(-C_2\left(\frac{R^2}{t}+\frac{t}{R^2}\right)\right)\le   \frac{\partial q_{B(x,R)}(t,x,\cdot)}{\partial n}(y) \le \frac{C_3R}{t^{d/2+1}}\exp\left(-C_4\left(\frac{R^2}{t}+\frac{t}{R^2}\right)\right).
\end{split}
\end{equation}

\end{lemma}
\begin{proof}
Due to the spatial homogeneous property of Brownian motion $\{B_t\}_{t\ge 0}$,  it holds that $$q_{B(x,R)}(t,x,y)=q_{B(0,R)}(t,0,y-x),\quad
t,R>0,x,y\in \R^d,$$  and so it suffices to prove \eqref{l2-1-1} for $x=0$.

According to \cite[Theorem 1.1]{Z} (for $d\ge 3$) and \cite[Theorem 1]{MS} (for $d\ge 1$), for any $t\in (0,1]$ and $x,y\in B(0,1)$,
\begin{equation}\label{l2-1-2a}
\begin{split}
 {c_1{t^{-d/2}}\left(\frac{(1-|x|)(1-|y|)}{t}\wedge 1\right)}e^{-\frac{c_2|x-y|^2}{t}}&\le q_{B(0,1)}(t,x,y)\\
 &
\le {c_3{t^{-d/2}}\left(\frac{(1-|x|)(1-|y|)}{t}\wedge 1\right)} e^{-\frac{c_4|x-y|^2}{t}};
\end{split}
\end{equation}while for any $ t\in (1,+\infty)$ and $x,y\in B(0,1)$,
\begin{equation}\label{l2-1-3}
c_1(1-|x|)(1-|y|)e^{-c_2t}\le q_{B(0,1)}(t,x,y)
\le c_3(1-|x|)(1-|y|)e^{-c_4t}.
\end{equation}

Using \eqref{l2-1-2a}, \eqref{l2-1-3} and the fact that $q_{B(0,1)}(t,0,y)=0$ for all
$y\in \partial B(0,1)$, we can obtain immediately that for all $y\in \partial B(0,1)$ and $t\in (0,1]$,
$$\frac{c_5}{t^{d/2+1}}\exp\left(-\frac{c_6}{t}\right)\le \frac{\partial q_{B(0,1)}(t,0,\cdot)}{\partial n}(y)\le \frac{c_7}{t^{d/2+1}}\exp\left(-\frac{c_8}{t}\right);$$
while for all $y\in \partial B(0,1)$ and $t\in (1,\infty)$,
$$ c_5\exp\left(-c_6t\right)\le \frac{\partial q_{B(0,1)}(t,0,\cdot)}{\partial n}(y)\le c_7\exp\left(-c_8 t\right).$$
Hence, it holds that for all  $t>0$ and $y\in \partial B(0,1)$,
\begin{equation}\label{l2-1-2}
\frac{c_9}{t^{d/2+1}}\exp\left(-c_{10}\left(\frac{1}{t}+t\right)\right)\le \frac{\partial q_{B(0,1)}(t,0,\cdot)}{\partial n}(y)
\le \frac{c_{11}}{t^{d/2+1}}\exp\left(-c_{12}\left(\frac{1}{t}+t\right)\right).
\end{equation}

On the other hand, thanks to the scaling property of the standard  Brownian motion $\{B_t\}_{t\ge0}$, for any $R>0$, $t>0$ and $y\in B(0,1)$,
\begin{equation*}
q_{B(0,1)}(t,0,y)=R^dq_{B(0,R)}(R^2t,0,Ry).
\end{equation*}

Combining this with \eqref{l2-1-2} yields that for every $t>0$ and $y\in \partial B(0,R)$,
\begin{align*}
\frac{\partial q_{B(0,R)}(t,0,\cdot)}{\partial n}(y)&=
R^{-(d+1)}\frac{\partial q_{B(0,1)}(R^{-2}t,0,\cdot)}{\partial n}(R^{-1}y)\\
&\le \frac{c_{11}}{R^{d+1}(R^{-2}t)^{d/2+1}}\exp\left(-c_{12}\left(\frac{1}{R^{-2}t}+R^{-2}t\right)\right)\\
&=\frac{c_{11}R}{t^{d/2+1}}\exp\left(-c_{12}\left(\frac{R^2}{t}+\frac{t}{R^2}\right)\right).
\end{align*}

Similarly, applying the first inequality in \eqref{l2-1-2}, we can obtain
that for every $t>0$ and $y\in \partial B(0,R)$,
\begin{align*}
\frac{\partial q_{B(0,R)}(t,0,\cdot)}{\partial n}(y)
\ge \frac{c_{13}R}{t^{d/2+1}}\exp\left(-c_{14}\left(\frac{R^2}{t}+\frac{t}{R^2}\right)\right).
\end{align*}
Therefore, the desired assertion \eqref{l2-1-1} follows.
\end{proof}

With aid of all the lemmas above, we have the following upper bound estimates for $p(t,x,y)$ when $|x-y|>C_0t^{1/2}$.

\begin{lemma}\label{l2-4}
Given any constant $C_0>0$, there exist positive constants $C_{5}$ and $C_{6}$ such that for all $x,y\in \R^d$ and $t>0$ with $|x-y|>2C_0t^{1/2}$,
\begin{equation}\label{l2-4-1}
   p(t,x,y)\le C_{5}t^{-d/2}\exp\left(-C_{6}\left(\frac{|x-y|^2}{t}+|x-y|\sqrt{\max\{g(|x|),g(|y|)\}}\right)\right).
\end{equation}
\end{lemma}
\begin{proof}
Without loss of generality, we assume that
$|x|\le |y|$.
In particular,
it holds that $|x-y|\le 2|y|$.
Now suppose that $|x-y|>2C_0t^{1/2}$. Set $U=B(y,{|x-y|}/{3})$. Then, according to \eqref{l2-4-2},
\begin{align*}
 p(t,x,y) 
 &=p(t,y,x)\\
&\le \Ee_y\left[\exp\left(-\tau_U \inf_{u\in U}V(u)\right)p\left(t-\tau_U,B_{\tau_U},x\right)\right]\\
&=\frac 1{2}\int_0^t \exp\left(-s \inf_{u\in U}V(u)\right)\cdot\left(\int_{\partial U}p(t-s,z,x)\frac{\partial q_U(s,y,\cdot)}{\partial n}(z)\,\sigma(dz)\right)\,ds\\
&\le c_1\int_0^t e^{-c_2sg(|y|)}(t-s)^{-d/2}e^{-\frac{c_2|x-y|^2}{t-s}}|\partial U|
\frac{|x-y|}{s^{d/2+1}}e^{-\frac{c_2|x-y|^2}{s}}\,ds\\
&\le c_3\int_0^t e^{-c_2sg(|y|)}(t-s)^{-d/2}e^{-\frac{c_2|x-y|^2}{t-s}}
\left(\frac{|x-y|^2}{s}\right)^{d/2}e^{-\frac{c_2|x-y|^2}{s}}s^{-1}ds\\
&\le c_4t^{-d/2-1}e^{-\frac{c_5|x-y|^2}{t}}\int_0^t \exp\left(-c_5\left(sg(|y|)+\frac{|x-y|^2}{s}\right)\right)ds.
\end{align*}
Here the second equality above is due to \eqref{l2-5-1}; in the second inequality we have used \eqref{l2-1-1}, and
the facts that $|x-y|>2C_0t^{1/2}$, $\inf_{u\in U}V(u)\ge c_6g(|y|)$ (which is deduced from \eqref{e1-2a}, \eqref{e1-3}
and  $|x-y|\le 2|y|$), and
$$p(t-s,z,x)\le q(t-s,z,x)\le c_7(t-s)^{-d/2}e^{-\frac{c_8|x-y|^2}{t-s}},\quad z\in \partial U;$$
the third inequality is due to $|\partial U|\le c_9|x-y|^{d-1}$; and the fourth inequality follows from
the fact that, thanks to $|x-y|>2C_0t^{{1}/{2}}$, for all $s\in (0,t)$,
\begin{align*}
&(t-s)^{-d/2}e^{-\frac{c_2|x-y|^2}{t-s}}
\left(\frac{|x-y|^2}{s}\right)^{d/2}e^{-\frac{c_2|x-y|^2}{s}}s^{-1}\\
&=|x-y|^{-d-2}\left(\frac{|x-y|^2}{t-s}\right)^{d/2}e^{-\frac{c_2|x-y|^2}{t-s}}
\left(\frac{|x-y|^2}{s}\right)^{d/2+1}
e^{-\frac{c_2|x-y|^2}{s}}\\
&\le c_{10} |x-y|^{-d-2} e^{-\frac{c_{11}|x-y|^2}{t-s}} e^{-\frac{c_{11}|x-y|^2}{s}} \le c_{12} t^{-d/2-1}e^{-\frac{c_{11}|x-y|^2}{t}}e^{-\frac{c_{11}|x-y|^2}{s}},
\end{align*}
where in the first inequality we have used
$$
\left(\frac{|z|^2}{s}\right)^{m} e^{-\frac{c_{2}|z|^2}{s}}\le e^{-\frac{c_{13}|z|^2}{s}},\quad s>0,\ m>0,\ z\in \R^d\ {\rm with}\ |z|>2C_0s^{1/2}$$
and the last inequality is again due to  $|x-y|>2C_0t^{{1}/{2}}$.

Now set $G(s):=sg(|y|)+\frac{|x-y|^2}{s}$ for all $s>0$. We can check directly that
$$
\inf_{s\in [0,t]}G(s)\ge
\begin{cases}
G(t)= tg(|y|)+\frac{|x-y|^2}{t}, &\ t\le \frac{|x-y|}{\sqrt{g(|y|)}},\\
G\left(\sqrt{\frac{|x-y|^2}{g(|y|)}}\right)=2|x-y|\sqrt{g(|y|)}, &\ t> \frac{|x-y|}{\sqrt{g(|y|)}}.
\end{cases}
$$
Putting this into the estimate above for $p(t,x,y)$, we arrive at
\begin{align*}
p(t,x,y)&\le
\begin{cases}
c_{14}t^{-d/2}\exp\left(-c_{15}\left(tg(|y|)+\frac{|x-y|^2}{t}\right)\right), &\ t\le \frac{|x-y|}{\sqrt{g(|y|)}},\\
c_{14}t^{-d/2}\exp\left(-c_{15}\left(|x-y|\sqrt{g(|y|)}+\frac{|x-y|^2}{t}\right)\right), &\ t> \frac{|x-y|}{\sqrt{g(|y|)}},
\end{cases}
\\
&\le c_{16}t^{-d/2}\exp\left(-c_{17}\left(|x-y|\sqrt{g(|y|)}+\frac{|x-y|^2}{t}\right)\right),
\end{align*}
where in the last inequality we used the fact that
$$tg(|y|)\le |x-y|\sqrt{g(|y|)}\le \frac{|x-y|^2}{t}\quad {\rm if}\,\, t\le \frac{|x-y|}{\sqrt{g(|y|)}}.$$ The proof is complete.
\end{proof}

In the rest of this section, we will give upper bound estimates for $T_t^V1(x)$.

\begin{lemma}\label{l2-2} There are constants $C_7,C_8>0$ such that for all $t>0$ and $x\in \R^d$ with $|x|\ge2$,
\begin{equation}\label{l2-2-1}
\begin{split}
T_t^V1(x)\le C_7\left(\exp(-C_8tg(|x|))+\exp\left(-\frac{C_8|x|^2}{t}\right)\right).
\end{split}
\end{equation}
\end{lemma}
\begin{proof}
According to \eqref{e1-4}, for any $x\in \R^d$ with $|x|\ge 2$,
\begin{align*}
T_t^V 1(x)&=\Ee_x\left[\exp\left(-\int_0^t V(B_s)\,ds\right)\I_{\{\tau_U>t\}}\right]+
\Ee_x\left[\exp\left(-\int_0^t V(B_s)\,ds\right)\I_{\{\tau_U\le t\}}\right]\\
&=:I_1(x)+I_2(x),
\end{align*}
where $U=B\left(x,{|x|}/{2}\right)$.

By \eqref{e1-2a} and \eqref{e1-3} we have
$\inf_{u\in U}V(u)\ge c_1g(|x|)$,
which implies immediately that
\begin{equation*}
I_1(x)\le \exp\left(-c_1tg(|x|)\right).
\end{equation*}

On the other hand,
\begin{align*}
I_2(x)\le& \Ee_x\left[\exp\left(-\int_0^t V(B_s)\,ds\right)\I_{\{\tau_U\le t,B_t\in B(x,{|x|}/{3})\}}\right]\\
&+\Ee_x\left[\exp\left(-\int_0^t V(B_s)\,ds\right)\I_{\{\tau_U\le t,B_t\notin B(x,{|x|}/{3})\}}\right]\\
=:&I_{21}(x)+I_{22}(x).
\end{align*}
By the strong Markov property,
\begin{align*}
I_{21}(x)&\le \Pp_x\left(\tau_U\le t, B_t\in B(x,{|x|}/{3})\right)\le \Ee_x\left[\Pp_{B_{\tau_U}}\left(B_{t-\tau_U}\in B(x,{|x|}/{3})\right)\right]\\
&\le \sup_{s\in [0,t],z\in \partial U}\int_{B(x,{|x|}/{3})}q(s,z,y)\,dy\\
&\le c_2\sup_{s\in [0,t]}\left( \left(\frac{|x|^2}{s}\right)^{{d}/{2}}\exp\left(-\frac{c_3|x|^2}{s}\right)\right)
\le c_4\exp\left(-\frac{c_5|x|^2}{t}\right),
\end{align*}
where in the fourth inequality we used \eqref{e1-6}.
At the same time, it holds that
\begin{align*}
I_{22}(x)&\le \Pp_x\left(B_t\notin B(x,{|x|}/{3})\right)=\int_{B(x,{|x|}/{3})^c}q(t,x,y)\,dy\\
&\le c_6t^{-d/2}\int_{B(x,{|x|}/{3})^c}\exp\left(-\frac{c_7|x-y|^2}{t}\right)\,dy\le c_{8}\exp\left(-\frac{c_{9}|x|^2}{t}\right),
\end{align*}
where the second inequality is due to \eqref{e1-6} and in the last step we have used the fact $|x|\ge 2$.

Combining all above estimates for $I_1(x)$ and $I_2(x)$ together yields the desired conclusion \eqref{l2-2-1}.
\end{proof}

\begin{lemma}\label{l2-3}
There exist positive constants $C_9, C_{10}$ such that for all $t>0$ and $x\in \R^d$,
\begin{equation}\label{l2-3-1}
T_t^V1(x)\le C_{9}e^{-C_{10}t}.
\end{equation}
\end{lemma}
\begin{proof} By \eqref{e1-4} and the fact that $V\ge0$, it is obvious that \eqref{l2-3-1} holds for all $0<t\le 1$.

Next, we consider the case that $t>1$.
Let $A:=\{x\in \R^d: V(x)\ge 1\}$. By
the local boundness of $V$
and \eqref{e1-3} we know $|A|>0$, and so
$\Pp_x(\Lambda)>0$ for every $x\in \R^d$, where
$$\Lambda:=\left\{\omega\in \Omega: B_t(\omega)\in A\ {\rm for\ all}\ t\in \left[1/2,1\right]\right\}.$$
Thus, for every $x\in \R^d$,
\begin{align*}
T_1^V1(x)&=\Ee_x\left[\exp\left(-\int_0^1 V(B_s)\,ds\right)\right]\\
&=\Ee_x\left[\exp\left(-\int_0^1 V(B_s)\,ds\right)\I_{\Lambda}\right]+\Ee_x\left[\exp\left(-\int_0^1 V(B_s)ds\right)\I_{\Lambda^c}\right]\\
&\le e^{-1/2}\Pp_x(\Lambda)+\Pp_x(\Lambda^c)<
\Pp_x(\Lambda)+\Pp_x(\Lambda^c)=1.
\end{align*}
Therefore, by the continuity of $x\mapsto T_1^V1(x)$, we obtain that for every compact set $K\subset \R^d$,
\begin{equation}\label{l2-3-2}
\sup_{x\in K}T_1^V(x)<1.
\end{equation}

On the other hand, according to \eqref{l2-2-1},
\begin{equation*}
\lim_{|x|\to \infty}T_1^V(x)=0.
\end{equation*}

Combining this with \eqref{l2-3-2} yields that there exists a positive constant $a\in (0,1)$ such that
$$
\sup_{x\in \R^d}T_1^V(x)=a<1.
$$

By applying the routine iteration arguments, we arrive at that for all $k\in \mathbb{N}_+$,
\begin{equation*}
\sup_{x\in \R^d}T_k^V1(x)\le a^k=e^{-(\log a^{-1}) k}.
\end{equation*} Furthermore, for any $t>1$,
\begin{equation*}
\sup_{x\in \R^d}T_t^V1(x)=\sup_{x\in \R^d}T_{t-[t]}^V\left(T_{[t]}^V1\right)(x)\le
\sup_{x\in \R^d}T_{[t]}^V1(x)\le e^{-(\log a^{-1}) [t]}\le a^{-1}e^{-(\log a^{-1}) t},
\end{equation*}
and so the proof is complete.
\end{proof}

Below, we define
\begin{equation}\label{e2-1}
t_0(s):=\frac{1+s}{\sqrt{g(s)}},\quad s>0.
\end{equation}
By applying Lemma \ref{l2-3}, we can obtain the following improvement of \eqref{l2-2-1}.

\begin{corollary}\label{E:cor}
For any constant $C_0>0$, there are constants $C_{11},C_{12}>0$ such that for all $t>0$ and $x\in \R^d$
\begin{equation}\label{r2-1-}
T_t^V 1(x)\le
 \begin{cases} C_{11}
\exp\left(-C_{12}tg(|x|)\right),\quad&0<t\le C_0t_0(|x|),\\
C_{11}\exp\left(-C_{12}(t+(1+|x|)\sqrt{g(|x|)})
\right),\quad &t\ge C_0t_0(|x|).\end{cases}
\end{equation}
\end{corollary}
\begin{proof}
When $|x|\le 2$,  $t_0(|x|)\le c_1$ and so we can deduce \eqref{r2-1-} directly from \eqref{l2-3-1}. So it suffices to prove \eqref{r2-1-} for
every $x\in \R^d$ with $|x|>2$.
According to \eqref{l2-2-1} and \eqref{l2-3-1}, we have
\begin{equation}\label{l4-1-2}
\begin{split}
T_t^V1(x)=&T_{{t}/{2}}^V\left(T_{{t}/{2}}^V 1\right)(x)\le \sup_{z\in \R^d}T_{{t}/{2}}^V1(z)\cdot
T_{{t}/{2}}^V1(x)\\
\le& c_1\left[\exp\left(-c_2(t+tg(|x|))\right)+\exp\left(-c_2\left(t +\frac{(1+|x|)^2}{t}\right)\right)\right]\\
\le&  c_3\left[\exp\left(-c_4tg(|x|)\right)+\exp\left(-c_2\left(t +\frac{(1+|x|)^2}{t}\right)\right)\right],\end{split}
\end{equation} where in the last inequality we used the fact that $g(r)\ge1$ for all $r\ge0$.
Furthermore, we note that for all $t\le C_0t_0(|x|)$, $tg(|x|)\le \frac{c_5(1+|x|)^2}{t}$.
Putting this fact into \eqref{l4-1-2}, we can get the desired assertion when $0<t\le C_0t_0(|x|)$. In particular, it
holds that
\begin{align*}
T^V_{\frac{C_0t_0(|x|)}{2}}(x)\le c_5\exp\left(-c_6(1+|x|)\sqrt{g(|x|)}\right).
\end{align*}
Combining this with \eqref{l2-3-1} yields that for every $t>C_0t_0(|x|)$,
\begin{align*}
T_t^V1(x)&=T_{\frac{C_0t_0(|x|)}{2}}^V\left(T_{t-\frac{C_0t_0(|x|)}{2}}^V 1\right)(x)\le \left(\sup_{z\in \R^d}T_{t-\frac{C_0t_0(|x|)}{2}}^V1(z)\right)\cdot
T_{\frac{C_0t_0(|x|)}{2}}^V1(x)\\
&\le c_7\exp\left(-c_8\left(t-\frac{C_0t_0(|x|)}{2}\right)\right)\exp\left(-c_8(1+|x|)\sqrt{g(|x|)}\right)\\
&\le c_9\exp\left(-c_{10}\left(t+(1+|x|)\sqrt{g(|x|)}\right)\right),
\end{align*}
where the last inequality we have used the fact that $t-\frac{C_0t_0(|x|)}{2}>\frac{t}{2}$ for all $t>C_0t_0(|x|)$. Therefore, the desired assertion follows.
\end{proof}

\section{Two-sided estimates when $0<t\le C_0t_0(|x|\wedge |y|)$ with any given $C_0>0$}\label{section3}
From this section, without loss of generality we always assume that $|x|\le |y|$. We start from the following on-diagonal estimates.
\begin{lemma}\label{l3-1}
For any given $C_0, C_0'>0$, there exist positive constants $C_i$, $1\le i\le 4$, such that for all $t>0$ and $x,y\in \R^d$ with $0<t\le C_0t_0(|x|)$ and $|x-y|\le C_0't^{{1}/{2}}$,
\begin{equation}\label{l3-1-1}
 C_1t^{-{d}/{2}}\exp\left(-C_2g(|x|)t\right)\le p(t,x,y)\le C_3t^{-{d}/{2}}\exp\left(-C_4g(|x|)t\right).
 \end{equation}
\end{lemma}
\begin{proof}
Suppose that $0<t\le C_0t_0(|x|)$ and $|x-y|\le C_0't^{1/2}$. By the semigroup property,
\begin{align*}
p(t,x,y)&=\int_{\R^d}p\left(t/2,x,z\right)p\left(t/2,z,y\right)dz\le c_1t^{-d/2}\int_{\R^d}p\left(t/2,x,z\right)dz\\
&=
c_1t^{-d/2}T_{t/2}^V 1(x)\le c_2t^{-d/2}\exp(-c_3tg(|x|)).
\end{align*}
Here in the first inequality we used the fact
\begin{align*}
p\left(t/2,z,y\right)\le q\left(t/2,z,y\right)\le c_1t^{-d/2},
\end{align*}
and the last inequality follows from \eqref{r2-1-} and the fact that $t/2\le C_0t_0(|x|)$.

On the other hand, set $U:=B(x,3C_0't^{1/2})$. Since $t_0(|x|)\le 1+|x|$,  $t^{1/2}\le C_0^{1/2}t_0(|x|)^{1/2}\le C_0^{1/2}(1+|x|)^{1/2}$. This along with
\eqref{e1-2a} and \eqref{e1-3} yields that
$\sup_{u\in U}V(u)\le c_4g(|x|)$.
Hence, for any $f\in C_b(\R^d)$ with ${\rm supp}[f]\subset B(x,2C_0't^{1/2})$, it holds that
\begin{align*}
T_t^Vf(x)&=\Ee_x\left[f(B_t)\exp\left(-\int_0^t V(B_s)\,ds\right)\right]
 \ge \Ee_x\left[f(B_t)\exp\left(-\int_0^t V(B_s)\,ds\right)\I_{\{t<\tau_U\}}\right] \\
&\ge \exp\left(-t\sup_{u\in U}V(u)\right)\Ee_x\left[f(B_t)\I_{\{t<\tau_U\}}\right]
 \ge \exp\left(-c_4tg(|x|)\right)\int_{B(x,2C_0't^{1/2})}q_U(t,x,z)f(z)\,dz\\
&\ge c_5t^{-d/2}\exp\left(-c_4tg(|x|)\right)\int_{B(x,2C_0't^{1/2})}f(z)\,dz,
\end{align*}
where the third inequality is due to ${\rm supp}[f]\subset B(x,2C_0't^{1/2})$, and in the last inequality we used the lower bounds \eqref{l2-1-2a} and \eqref{l2-1-3} as well as the scaling property of Dirichlet heat kernel $q_U(t,x,z)$.
Therefore, by using the inequality above we can deduce the lower bound of $p(t,x,y)$ in \eqref{l3-1-1}.
\end{proof}

\begin{lemma}\label{l3-2}
For any positive constants $C_0, C_0'$, there exist positive constants $C_{5}$ and $C_{6}$ such that for all $t>0$ and $x,y\in \R^d$ with  $0<t\le C_0t_0(|x|)$ and $|x-y|> 2C_0't^{1/2}$,
\begin{equation}\label{l3-2-1}
 p(t,x,y)\ge C_{5}t^{-d/2}\exp\left(-C_{6}\left(g(|x|)t+|x-y|\sqrt{g(|y|)}+\frac{|x-y|^2}{t}\right)\right).
\end{equation}
\end{lemma}
\begin{proof}
Suppose that $0<t\le C_0t_0(|x|)$ and $|x-y|> 2C_0't^{1/2}$. Define
$U=B(y,|x-y|-C_0't^{1/2})$ and $W=B(x,3C_0't^{1/2})$. One can verify that
\begin{equation}\label{l3-2-2}
|\partial U \cap W|\ge
c_1t^{\frac{d-1}{2}}.
\end{equation}
By \eqref{l2-4-2},
we have
\begin{equation}\label{l3-2-3}
\begin{split}
p(t,x,y)
&=\Ee_y\left[\exp\left(-\int_0^{\tau_U}V(B_s)\,ds\right)p\left(t-\tau_U,B_{\tau_U},x\right)\I_{\{\tau_U\le t\}}\right]\\
&\ge \Ee_y\left[\exp\left(-\tau_U \cdot \sup_{u\in U}V(u)\right)p\left(t-\tau_U,B_{\tau_U},x\right)
\I_{\{\tau_U\le t,B_{\tau_U}\in W\}}\right]\\
&=\frac 1{2}\int_0^t \exp\left(-s\sup_{u\in U}V(u)\right)\left(\int_{\partial U \cap W}p(t-s,z,x)\frac{\partial q_U(s,y,\cdot)}{\partial n}(z)\,\sigma(dz)\right)\,ds\\
&\ge \frac{c_2|x-y|}{t^{1/2}}\int_0^t e^{-c_3sg(|y|)}(t-s)^{-d/2}e^{-c_3(t-s)g(|x|)}\left(\frac{t}{s}\right)^{d/2}e^{-\frac{c_3|x-y|^2}{s}}s^{-1}\,ds\\
&\ge c_4 t^{-d/2}e^{-c_5tg(|x|)}\int_0^{t/2}\exp\left(-c_5\left(sg(|y|)+\frac{|x-y|^2}{s}\right)\right)s^{-1}\,ds.
\end{split}
\end{equation}
Here the second equality is due to \eqref{l2-5-1}, in the second inequality above we have used \eqref{l2-1-1}
(note that here $|x-y|>2C_0't^{1/2}$) and \eqref{l3-2-2}, as well as the facts that
\begin{align*}
&\sup_{u\in U}V(u)\le \sup_{u:|u-y|\le |x-y|}V(u)\le \sup_{u:|u|\le 2|y|+|x|\le 3|y|}V(u)\le c_6g(|y|),\\
&\sup_{z\in W}p(t-s,z,x)\ge c_7(t-s)^{-d/2}e^{-c_8(t-s)g(|x|)},
\end{align*}
which can be deduced from \eqref{e1-2a}, \eqref{e1-3} and \eqref{l3-1-1} respectively, and the last inequality follows from the
fact that $\frac{|x-y|}{t^{1/2}}\ge c_9$.

Furthermore, we can deduce that
\begin{equation}\label{l3-2-4}
\begin{split}
&\int_0^{t/2}\exp\left(-c_5\left(sg(|y|)+\frac{|x-y|^2}{s}\right)\right)s^{-1}\,ds\\
&\ge
\begin{cases}
\displaystyle\int_{{t}/{4}}^{t/2}\exp\left(-c_5\left(sg(|y|)+\frac{|x-y|^2}{s}\right)\right)s^{-1}\,ds,\ & t\le \frac{4|x-y|}{\sqrt{g(|y|)}},\\
\displaystyle\int_{\frac{|x-y|}{\sqrt{g(|y|)}}}^{\frac{2|x-y|}{\sqrt{g(|y|)}}}\exp\left(-c_5\left(sg(|y|)+\frac{|x-y|^2}{s}\right)\right)s^{-1}\,ds,\ & t> \frac{4|x-y|}{\sqrt{g(|y|)}},\\
\end{cases}
\\
&\ge
\begin{cases}
c_{10}\exp\left(-c_{11}\left(tg(|y|)+\frac{|x-y|^2}{t}\right)\right),\ & t\le \frac{4|x-y|}{\sqrt{g(|y|)}},\\
c_{10}\exp\left(-c_{11}|x-y|\sqrt{g(|y|)}\right),\ & t> \frac{4|x-y|}{\sqrt{g(|y|)}},\\
\end{cases}
\\
&\ge
\begin{cases}
c_{12}\exp\left(-\frac{c_{13}|x-y|^2}{t}\right),\ & t\le \frac{4|x-y|}{\sqrt{g(|y|)}},\\
c_{12}\exp\left(-c_{13}|x-y|\sqrt{g(|y|)}\right),\ & t> \frac{4|x-y|}{\sqrt{g(|y|)}},\\
\end{cases}
\end{split}
\end{equation}
where the last inequality follows from the fact that
\begin{align*}
& tg(|y|)\le \frac{16|x-y|^2}{t}\quad \ {\rm for}\   t\le \frac{4|x-y|}{\sqrt{g(|y|)}}.
\end{align*}
Putting \eqref{l3-2-4} into \eqref{l3-2-3}, we obtain
\begin{align*}
p(t,x,y)&\ge
\begin{cases}
c_{14}t^{-d/2} \exp\left(-c_{15}\left(tg(|x|)+\frac{|x-y|^2}{t}\right)\right), & t\le \frac{4|x-y|}{\sqrt{g(|y|)}},\\
c_{14}t^{-d/2} \exp\left(-c_{15}\left(tg(|x|)+|x-y|\sqrt{g(|y|)}\right)\right),\ & t> \frac{4|x-y|}{\sqrt{g(|y|)}},\\
\end{cases}
\\
&\ge c_{16}t^{-d/2}\exp\left(-c_{17}\left(tg(|x|)+|x-y|\sqrt{g(|y|)}+\frac{|x-y|^2}{t}\right)\right).
\end{align*}
By now we finish the proof for \eqref{l3-2-1}.
\end{proof}

\begin{lemma}\label{l3-3}
For any positive constants $C_0, C_0'$, there exist positive constants $C_{7}$ and $C_{8}$ such that for all $t>0$ and $x,y\in \R^d$ with  $0<t\le C_0t_0(|x|)$ and $|x-y|> 2C_0't^{1/2}$,
\begin{equation}\label{l3-3-1}
 p(t,x,y)\le C_{7}t^{-d/2}\exp\left(-C_{8}\left(g(|x|)t+|x-y|\sqrt{g(|y|)}+\frac{|x-y|^2}{t}\right)\right).
\end{equation}
\end{lemma}
\begin{proof}
Suppose that $0<t\le C_0t_0(|x|)$ and $|x-y|> 2C_0't^{1/2}$. The proof is split into two cases.

\indent{\bf Case 1:} $|x-y|\le {|y|}/{4}$. By the semigroup property we have
\begin{align*}
  p(t,x,y)
&=\int_{\{z:|z-y|\le {|x-y|}/{2}\}}p\left(t/2,x,z\right)p\left(t/2,z,y\right)dz\\
&\quad +
\int_{\{z:|z-y|> {|x-y|}/{2}\}}p\left(t/2,x,z\right)p\left(t/2,z,y\right)dz\\
&=:I_1+I_2.
\end{align*}

When $|z-y|\le {|x-y|}/{2}$, it holds that
\begin{equation*}
|z-x|\ge |x-y|-|z-y|\ge \frac{|x-y|}{2}\ge C_0't^{1/2}.
\end{equation*}
Then, by \eqref{l2-4-1},
\begin{align*}
p\left(t/2,x,z\right)&\le c_1\exp\left(-c_2\left(|x-z|\sqrt{\max\{g(|x|),g(|z|)\}}+\frac{|x-z|^2}{t}\right)\right)\\
&\le c_3\exp\left(-c_4\left(|x-y|\sqrt{g(|y|)}+\frac{|x-y|^2}{t}\right)\right).
\end{align*}
Here in the last inequality follows from
\begin{align*}
g(|z|)\ge g\left(|y|-|z-y|\right)\ge g\left(|y|-\frac{|x-y|}{2}\right)\ge g\left(\frac{7|y|}{8}\right)\ge c_5g(|y|),
\end{align*}
where we have  used \eqref{e1-2a} and the condition $|x-y|\le {|y|}/{4}$.
Hence, we obtain
\begin{equation}\label{l3-3-2a}
\begin{split}
I_1&\le \sup_{z\in \R^d:|z-y|\le {|x-y|}/{2}}p\left(t/2,x,z\right)\cdot \int_{\R^d}p\left(t/2,z,y\right)\,dz\\
&\le c_3\exp\left(-c_4\left(|x-y|\sqrt{g(|y|)}+\frac{|x-y|^2}{t}\right)\right)\cdot T_{t/2}^V 1(y)\\
&\le c_6\exp\left(-c_7\left(|x-y|\sqrt{g(|y|)}+\frac{|x-y|^2}{t}+\min\left\{g(|y|)t,
(t+(1+|y|))\sqrt{g(|y|)}\right\}\right)\right)\\
&\le c_6\exp\left(-c_7\left(|x-y|\sqrt{g(|y|)}+\frac{|x-y|^2}{t}+\min\left\{g(|x|)t,
(t+(1+|x|))\sqrt{g(|x|)}\right\}\right)\right)\\
&\le c_8\exp\left(-c_9\left(|x-y|\sqrt{g(|y|)}+\frac{|x-y|^2}{t}+g(|x|)t\right)\right),
\end{split}
\end{equation}
where the third inequality follows from \eqref{r2-1-}, the fourth inequality is due to the fact $|x|\le |y|$, and the last
one follows from $0<t\le C_0t_0(|x|)$.

Similarly, when $|z-y|>{|x-y|}/{2}>C_0't^{1/2}$, we get from  \eqref{l2-4-1} that
\begin{align*}
p\left(t/2, y,z\right)& \le c_{10}\exp\left(-c_{11}\left(|y-z|\sqrt{\max\{g(|y|),g(|z|)\}}+\frac{|y-z|^2}{t}\right)\right)\\
&\le  c_{12}\exp\left(-c_{13}\left(|x-y|\sqrt{g(|y|)}+\frac{|x-y|^2}{t}\right)\right).
\end{align*}
Therefore, using \eqref{r2-1-} again and following the same arguments for the estimate of $I_1$, we have
\begin{align*}
I_2& \le \sup_{z\in \R^d:|z-y|\le {|x-y|}/{2}}p\left(t/2,y,z\right)\cdot \int_{\R^d}p\left(t/2,x,z\right)\,dz\\
&\le c_{14}\exp\left(-c_{15}\left(|x-y|\sqrt{g(|y|)}+\frac{|x-y|^2}{t}+g(|x|)t\right)\right).
\end{align*}

Combining with both estimates for $I_1$ and $I_2$, we can prove \eqref{l3-3-1} under the condition $|x-y|\le {|y|}/{4}$.

{\bf Case 2:} $|x-y|> {|y|}/{4}$. Since $|x|\le |y|$, it holds that
\begin{equation}\label{l3-3-2}
{|y|}/{4}\le |x-y|\le |x|+|y|\le 2|y|.
\end{equation}
Now we set
\begin{align*}
p(t,x,y)&=\int_{\{z:|z-y|\le {|y|}/{16}\}}p\left(t/2,x,z\right)p\left(t/2,z,y\right)\,dz\\
&\quad +
\int_{\{z:|z-y|> {|y|}/{16}\}}p\left(t/2,x,z\right)p\left(t/2,z,y\right)\,dz\\
&=:J_1+J_2.
\end{align*}

When $|z-y|\le {|y|}/{16}$, thanks to \eqref{l3-3-2}, it holds that
\begin{align*}
&|z-x|\ge |x-y|-|z-y|\ge |x-y|-{|y|}/{16}\ge |x-y|-{|x-y|}/{4}\ge \frac{3C_0't^{1/2}}{4},\\
&|z|\ge |y|-|z-y|\ge |y|-{|y|}/{16}\ge {|y|}/{2}.
\end{align*}
According to this and \eqref{l2-4-1}, we derive that
\begin{align*}
p\left(t/2,x,z\right)&\le c_{16}\exp\left(-c_{17}\left(|x-z|\sqrt{\max\{g(|x|),g(|z|)\}}+\frac{|x-z|^2}{t}\right)\right)\\
&\le c_{18}\exp\left(-c_{19}\left(|x-y|\sqrt{g(|y|)}+\frac{|x-y|^2}{t}\right)\right).
\end{align*}
This further along with the arguments for \eqref{l3-3-2a} yields that
\begin{align*}
J_1\le c_{20}\exp\left(-c_{21}\left(|x-y|\sqrt{g(|y|)}+\frac{|x-y|^2}{t}+g(|x|)t\right)\right).
\end{align*}

On the other hand, when $|z-y|\ge {|y|}/{16}$, by \eqref{l3-3-2} it holds that
\begin{equation*}
|z-y|\ge \frac{|y|}{16}\ge \frac{|x-y|}{32}\ge \frac{C_0't^{1/2}}{16}.
\end{equation*}
So, applying \eqref{l2-4-1} again, we find that
\begin{align*}
p\left(t/2,z,y\right)&\le c_{22}\exp\left(-c_{23}\left(|z-y|\sqrt{\max\{g(|y|),g(|z|)\}}+\frac{|z-y|^2}{t}\right)\right)\\
&\le c_{24}\exp\left(-c_{25}\left(|x-y|\sqrt{g(|y|)}+\frac{|x-y|^2}{t}\right)\right).
\end{align*}
Using this and following the same arguments as these for $I_2$, we arrive at
\begin{equation*}
J_2\le c_{26}\exp\left(-c_{27}\left(|x-y|\sqrt{g(|y|)}+\frac{|x-y|^2}{t}+g(|x|)t\right)\right).
\end{equation*}

Putting the estimates for $J_1$ and $J_2$ together, we can prove the desired conclusion \eqref{l3-3-1} when $|x-y|> {|y|}/{4}$.
So, we complete the proof.
\end{proof}

Summarising results from Lemmas \ref{l3-1}--\ref{l3-3},  we can obtain  the following two-sided estimates for
$p(t,x,y)$ when $0<t\le C_0t_0(|x|)$.
\begin{proposition}\label{p3-1}
For any $C_0>0$, there exist positive constants $C_{i}$, $9\le i\le 12$, such that
for all $x,y\in \R^d$ with $|x|\le |y|$ and all $0<t\le C_0t_0(|x|)$,
\begin{equation}\label{p3-1-1}
\begin{split}
&C_{9}t^{-d/2}\exp\left(-C_{10}\left(tg(|x|)+|x-y|\sqrt{g(|y|)}+\frac{|x-y|^2}{t}\right)\right)\\
&\le p(t,x,y)\\
&\le C_{11}t^{-d/2}\exp\left(-C_{12}\left(tg(|x|)+|x-y|\sqrt{g(|y|)}+\frac{|x-y|^2}{t}\right)\right).
\end{split}
\end{equation}
\end{proposition}
\begin{proof}
According to \eqref{l3-1-1}, \eqref{l3-2-1} and \eqref{l3-3-1}, it suffices to prove that
for any given $c_1>1$, $x,y\in \R^d$ and $0\le t\le C_0t_0(|x|)$ with $|x-y|\le c_1t^{1/2}$, it holds that
\begin{equation}\label{p3-1-2}
|x-y|\sqrt{g(|y|)}+\frac{|x-y|^2}{t}\le c_2\left(1+tg(|x|)\right).
\end{equation}
Since $|x-y|\le c_1t^{1/2}$, $\frac{|x-y|^2}{t}\le c_3$.
Moreover, according to Young's inequality,
\begin{align*}
|x-y|\sqrt{g(|y|)}&\le c_4t^{1/2}\sqrt{g(|x|)}\le \frac{c_4}{2}\left(1+tg(|x|)\right).
\end{align*}
Here the first inequality follows from
\begin{align*}
g(|y|)&\le g\left(|x|+|x-y|\right)\le g\left(|x|+c_1t^{1/2}\right)\le
g\left(|x|+c_5\sqrt{t_0(|x|)}\right)\\
&\le g(|x|+c_5(1+|x|)^{1/2})\le c_6g(|x|),
\end{align*}
where the third inequality is due to the fact $0<t\le C_0t_0(|x|)$, in the fourth inequality we used  $t_0(|x|)\le 1+|x|$,
and in the last inequality we have used \eqref{e1-2a}.

Combining with all the estimates above together yields \eqref{p3-1-2}.
Therefore, the proof is complete.
\end{proof}

\section{Two-sided bounds when $t>C_0t_0(|x|\wedge |y|)$ with any $C_0>0$}\label{section4}
Recall that $t_0(s)=\frac{1+s}{\sqrt{g(s)}}$ for every $s\ge0$. We say $t_0:[0,\infty)\to \R_+$ is almost increasing (resp.\ decreasing), if
there exists an increasing (resp.\ a decreasing) function $h:[0,+\infty) \to [1,+\infty)$ such that for some
positive constants $C_*,C^*$,
\begin{equation}\label{e4-1}
C_*h(s)\le t_0(s)\le C^*h(s), \quad s\ge0.
\end{equation}
In this section we will present two-sided estimates for $p(t,x,y)$ when $t>C_0t_0(|x|\wedge |y|)$ with any $C_0>0$, according to
different (almost monotone) properties for $t_0(\cdot)$. Again we will assume that $|x|\le |y|$.

\subsection{$t_0(\cdot)$ is almost decreasing}

\begin{lemma}\label{l4-1}
Assume that $t_0(\cdot)$ is almost decreasing. For any $C_0>0$, there exist positive constants $C_{1}, C_{2}$ such that for all $x,y\in \R^d$ and $t>C_0t_0(|x|)$,
\begin{equation}\label{l4-1-1}
p(t,x,y)\le C_{1}e^{-C_{2}t}\exp\left(-C_{2}(1+|y|)\sqrt{g(|y|)}\right).
\end{equation}
\end{lemma}
\begin{proof}
Since $|x|\le |y|$, by \eqref{e4-1} and the almost decreasing property of $t_0(\cdot)$, we can find
a constant $c_1>0$ independent of $x,y$ such that
\begin{equation}\label{l4-1-1a}
t_0(|x|)\ge c_1t_0(|y|).
\end{equation}
Therefore,
\begin{equation}\label{l4-1-3}
\begin{split}
 p(t,x,y)
&=\int_{\R^d}p\left(t/2,x,z\right)p\left(t/2,z,y\right)dz\\
&\le \sup_{z\in \R^d}p\left(t/2,x,z\right)\int_{\R^d}p\left(t/2,z,y\right)dz=
\sup_{z\in \R^d}p\left(t/2,x,z\right)T_{t/2}^V1(y)\\
&\le c_{2}t^{-d/2}e^{-c_{3}t}\exp\left(-c_{4}(1+|y|)\sqrt{g(|y|)}\right).
\end{split}
\end{equation}
Here in the second inequality above we used \eqref{r2-1-} for $T_{t/2}^V 1(y)$, as well as
the fact that $t>C_0t_0(|x|)\ge C_0c_1t_0(|y|)$ ensured by \eqref{l4-1-1a}, and
\begin{equation*}
\sup_{z\in \R^d}p\left(t/2,x,z\right)\le \sup_{z\in \R^d}q\left(t/2,x,z\right)\le c_{5}t^{-d/2}.
\end{equation*}
The inequality above, along with the fact that for all $t>C_0t_0(|x|)$
\begin{align*}
t^{-d/2}&\le (C_0t_0(|x|))^{-d/2}\le c_{6}t_0(|y|)^{-d/2}
\le c_{7}\left(\frac{\sqrt{g(y)}}{1+|y|}\right)^{d/2}\\
&\le c_8\left((1+|y|)\sqrt{g(|y|)}\right)^{d/2}\le c_{9}\exp\left(\frac{c_{4}}{2}(1+|y|)\sqrt{g(|y|)}\right)
\end{align*}
with  $c_{4}$ being the constant in \eqref{l4-1-3},
yields the desired assertion \eqref{l4-1-1}.
\end{proof}

\begin{lemma}\label{l4-2}
Assume that $t_0(\cdot)$ is almost decreasing. For any $C_0>0$, there exist positive constants $C_{3}, C_{4}$ such that for all $x,y\in \R^d$ and $t>C_0t_0(|x|)$,
\begin{equation}\label{l4-2-1}
p(t,x,y)\ge C_{3}e^{-C_{4}t}\exp\left(-C_{4}(1+|y|)\sqrt{g(|y|)}\right).
\end{equation}
\end{lemma}
\begin{proof} The proof is split into three cases.

{\bf Case 1:} $|x|\ge 4$.
According to    \eqref{e1-4}, for every $f\in C_b(\R^d)$ with ${\rm supp}[f]\subset B(0,5)$ and $u\in B(0,5)$,
\begin{align*}
T_t^Vf(u)&=\Ee_u\left[f(B_t)\exp\left(-\int_0^t V(B_s)\,ds\right)\right]\\
&\ge \Ee_u\left[f(B_t)\exp\left(-\int_0^t V(B_s)\,ds\right)\I_{\{t>\tau_{B(0,6)}\}}\right] \ge e^{-c_1t}\int_{B(0,5)}q_{B(0,6)}(t,u,z)f(z)\,dz,
\end{align*}
where we have used the fact that $\sup_{z\in B(0,6)}V(z)\le c_1$. This implies immediately that for any $u,z\in B(0,5)$,
\begin{equation*}
p(t,u,z)\ge e^{-c_1t}q_{B(0,6)}(t,u,z).
\end{equation*}
Then, by \eqref{l2-1-2a} and \eqref{l2-1-3}, for all $u,z\in B(0,5)$ and $t>0$,
\begin{equation}\label{l4-2-2}
\begin{split}
p(t,u,z)\ge&
\begin{cases}
c_2t^{-d/2}e^{-\frac{c_3}{t}},\ &\ 0<t\le 1\\
c_2e^{-c_3t},\ &\ t>1
\end{cases}\\
\ge &c_4\exp\left(-c_5\left(t+\frac{1}{t}\right)\right).
\end{split}
\end{equation}

Now set $U:=B(x,|x|-2)$ and $W:=B(0,3)$. So we obtain for all $u\in B(0,2)$,
\begin{equation}\label{l4-2-2a}
\begin{split}
  p(t,x,u)
&=\Ee_x\left[\exp\left(-\int_0^{\tau_U}V(B_s)\,ds\right)\I_{\{\tau_U\le t\}}
p\left(t-\tau_U,B_{\tau_U},u\right)
\right]\\
&\ge \Ee_x\left[\exp\left(-\tau_U \cdot \sup_{z\in U}V(z)\right)p\left(t-\tau_U,B_{\tau_U},u\right)
\I_{\{\tau_U\le t,B_{\tau_U}\in W\}}\right]\\
&= \frac{1}{2}\int_0^t \exp\left(-s\sup_{z\in U}V(z)\right)\left(\int_{\partial U \cap W}p(t-s,z,u)\frac{\partial q_U(s,x,\cdot)}{\partial n}(z)\,\sigma(dz)\right)\,ds\\
&\ge c_6\exp\left(-c_7\left(t+\frac{1}{t}\right)\right)
\int_0^{t/2}\exp\left(-c_7\left(sg(|x|)+\frac{|x|^2}{s}+\frac{s}{|x|^2}\right)\right)s^{-(d/2+1)}\,ds\\
&\ge c_6\exp\left(-c_7\left(t+\frac{1}{t}\right)\right)\int_{\frac{C_0t_0(|x|)}{8}}^{\frac{C_0t_0(|x|)}{4}}
\exp\left(-c_7\left(sg(|x|)+\frac{|x|^2}{s}+\frac{s}{|x|^2}\right)\right)s^{-(d/2+1)}\,ds\\
&\ge c_8\exp\left(-c_7\left(t+\frac{1}{t}\right)\right)\exp\left(-c_{9}|x|\sqrt{g(|x|)}\right)\\
&\ge c_{10}e^{-c_{7}t}\exp\left(-c_{11}(1+|x|)\sqrt{g(|x|)}\right).
\end{split}
\end{equation}
Here the first step is due to \eqref{l2-4-2}, the third step is a consequence of \eqref{l2-5-1}, the fourth
step
follows from \eqref{l2-1-1},
\eqref{l4-2-2}, as well as the
the property that $|\partial U \cap W|\ge c_{12}$, $\sup_{z\in U}V(z)\le c_{13}g(|x|)$,
in the sixth step we have used the fact that
$\Big($by noting $t_0(|x|)=\frac{1+|x|}{\sqrt{g(|x|)}}$$\Big)$,
\begin{align*}
& \int_{\frac{C_0t_0(|x|)}{8}}^{\frac{C_0t_0(|x|)}{4}}
\exp\left(-c_7\left(sg(|x|)+\frac{|x|^2}{s}+\frac{s}{|x|^2}\right)\right)s^{-(d/2+1)}\,ds\\
&\ge c_{14}t_0(|x|)^{-d/2}\cdot\inf_{s\in [\frac{C_0t_0(|x|)}{8},\frac{C_0t_0(|x|)}{4}]}\exp\left(-c_7\left(sg(|x|)+\frac{|x|^2}{s}+\frac{s}{|x|^2}\right)\right)\\
&\ge c_{15} \left((1+|x|)\sqrt{g(|x|)}\right)^{-d/2}\exp\left(-c_{16}|x|\sqrt{g(|x|)}\right)
\ge c_{17}\exp\left(-c_{18}|x|\sqrt{g(|x|)}\right)
\end{align*}
and the last inequality follows from the fact  that for all $C_0t_0(|x|)<t\le 1$ and $|x|\ge 4$,
$$
e^{-\frac{c_7}{t}}\ge e^{-\frac{c_{19}}{t_0(|x|)}}\ge \exp\left(-\frac{c_{20}\sqrt{g(|x|)}}{1+|x|}\right)\ge \exp\left(-c_{21}(1+|x|)\sqrt{g(|x|)}\right).
$$
In particular, we arrive at that for every $|x|\ge 4$, $u\in B(0,2)$ and $t>C_0t_0(|x|)$,
\begin{equation}\label{l4-2-3}
p(t,x,u)\ge c_{10}e^{-c_{7}t}\exp\left(-c_{11}(1+|x|)\sqrt{g(|x|)}\right).
\end{equation}

Noting that $|y|\ge |x|\ge 4$, we know from \eqref{l4-1-1a} that
\eqref{l4-2-3} still holds with $x$ replaced by $y$. Hence,
\begin{align*}
p(t,x,y)&\ge \int_{B(0,2)}p\left(t/2,x,u\right)p\left(t/2,u,y\right)\,du\ge c_{22}\inf_{u\in B(0,2)}p\left(t/2,x,u\right)\cdot \inf_{u\in B(0,2)}p\left(t/2,u,y\right)\\
&\ge c_{23}e^{-c_{24}t}\exp\left(-c_{24}(1+|y|)\sqrt{g(|y|)}\right).
\end{align*}
Thus, we obtain  \eqref{l4-2-1} under the condition that $|x|\ge 4$.

{\bf Case 2: $|x|\le 4$ and $|y|\ge 4$.} Since $|y|\ge 4$, \eqref{l4-2-3} holds with $x$ replaced by $y$ here. On the other hand, it is noted that
\eqref{l4-2-2} holds for $u,z\in B(0,2)$. Hence we have
\begin{align*}
p(t,x,y)&\ge \int_{B(0,2)}p\left(t/2,x,u\right)p\left(t/2,u,y\right)\,du\\
&\ge c_{25}\inf_{u\in B(0,2)}p\left(t/2,x,u\right)\cdot \inf_{u\in B(0,2)}p\left(t/2,u,y\right)\\
&\ge c_{26}
\exp\left(-c_{27}\left(t+\frac{1}{t}\right)\right)
\exp\left(-c_{27}(1+|y|)\sqrt{g(|y|)}\right)\\
&\ge c_{28}e^{-c_{29}t}\exp\left(-c_{29}(1+|y|)\sqrt{g(|y|)}\right),
\end{align*}
where   the last inequality is due to the fact that $t\ge \inf_{u\in B(0,4)}t_0(|u|)\ge c_{30}>0$. Therefore, we proved \eqref{l4-2-1} for $|x|\le 4$, $|y|\ge 4$.

{\bf Case 3:  $|x|\le 4$ and $|y|\le  4$.}
Note that by definition we have $\inf_{u\in B(0,4)}t_0(|u|)\ge c_{30}>0$, and so the assumption
$t>C_0t_0(|x|)$ implies that there exists a positive constant $c_{31}$ such that $t>c_{31}$.
As mentioned above, \eqref{l4-2-2} holds for $u,z\in B(0,4)$ and so it holds for $x,y$. Thus,
\begin{align*}
p(t,x,y)&\ge
c_{32}\exp\left(-c_{33}\left(t+\frac{1}{t}\right)\right)\\
&\ge c_{34}e^{-c_{35}t}\ge
c_{36}e^{-c_{37}t}\exp\left(-c_{37}(1+|y|)\sqrt{g(|y|)}\right),
\end{align*}
where in the second inequality we have used that $e^{-\frac{c_{33}}{t}}\ge c_{38}e^{-c_{35}t}$ for all $t>c_{31}$.
\end{proof}

\subsection{$t_0(\cdot)$ is almost increasing}
\begin{lemma}\label{l4-3}
Assume that $t_0(\cdot)$ is almost increasing. For any $C_0>0$, there are positive constants $C_{5}$, $C_{6}$ such that for all $x,y\in \R^d$ and $t>C_0t_0(|x|)$,
\begin{equation}\label{l4-3-1}
\begin{split}
p(t,x,y)&\le C_{5}
e^{-C_{6}t}\min\left\{\exp\left(-\frac{C_{6}(1+|y|)^2}{t}\right),\exp\left(-C_{6}(1+|y|)\sqrt{g(|y|)}\right)\right\}.
\end{split}
\end{equation}
\end{lemma}
\begin{proof} The proof is split into three cases.

{\bf Case 1:  $|y|\ge 2$ and $|y|\ge 2|x|$.}
When $|y|\ge 2$ and $|y|\ge 2|x|$, by the condition that $t_0(\cdot)$ is almost increasing we can find a constant $C_0'>0$ such that $C_0't_0(|y|)> C_0t_0(|x|)$.
We first consider the case $C_0t_0(|x|)< t\le C_0't_0(|y|)$.

Note that for every $z\in \R^d$ satisfying $|z-y|\le |x-y|/2$, it holds that
\begin{align}\label{l4-3-2}
|z-x|\ge |y-x|/2\ge \frac{|y|-|x|}{2}\ge \frac{|y|}{4}\ge c_1t_0(|y|)\ge c_2\max\{t,1\}\ge 2c_3t^{1/2}.
\end{align}
where we have used the facts that $t_0(y)\le 1+|y|\le 2|y|$, $t\le C_0't_0(|y|)$ and  $\inf_{z\in \R^d}t_0(|z|)>0$ thanks to the fact that $t_0(\cdot)$ almost increasing.
Hence we can apply \eqref{l2-4-1} to obtain that for all $z\in \R^d$ with $|z-y|\le |x-y|/2$,
\begin{equation}\label{l4-3-3}
\begin{split}
p\left(t/2,x,z\right)&\le c_4t^{-d/2}\exp\left(-c_5\left(|x-z|\sqrt{\max\{g(|x|),g(|z|)\}}+\frac{|x-z|^2}{t}\right)\right)\\
&\le c_6\exp\left(-c_{7}\left(|y|\sqrt{g(|y|)}+\frac{|y|^2}{t}\right)\right)\le c_{8}\exp\left(-\frac{c_{9}(1+|y|)^2}{t}\right).
\end{split}
\end{equation}
Here the second inequality follows from \eqref{l4-3-2}, the facts that $t>C_0t_0(|x|)\ge c_{10}$ and,  for all $z\in \R^d$ with $|z-y|\le |x-y|/2$,
\begin{align*}
g(|z|)\ge c_{11}g\left(|y|-|z-y|\right)\ge c_{12}g\left(|y|-\frac{|y|+|x|}{2}\right)\ge c_{13}g(|y|),
\end{align*}
where we have also used \eqref{e1-2a}, and the last inequality is a consequence of the property that
\begin{align*}
\frac{|y|^2}{t}\ge c_{14}|y|\sqrt{g(|y|)},\quad t\le C_0't_0(|y|)\hbox{ and } |y|\ge 2.
\end{align*}

While for every $z\in \R^d$ with $|z-y|> |x-y|/2$, according to \eqref{l4-3-2} we can still apply \eqref{l2-4-1}
follow the same arguments above to derive that
\begin{equation}\label{l4-3-4}
\begin{split}
p\left(t/2,z,y\right)&\le c_{15}t^{-d/2}\exp\left(-c_{16}\left(|y-z|\sqrt{\max\{g(|y|),g(|z|)\}}+\frac{|y-z|^2}{t}\right)\right)\\
&\le c_{17}\exp\left(-c_{18}\left(|y|\sqrt{g(|y|)}+\frac{|y|^2}{t}\right)\right)\\
&\le c_{19}\exp\left(-\frac{c_{20}(1+|y|)^2}{t}\right).
\end{split}
\end{equation}

Now, write
\begin{align*}
p(t,x,y)
&=\int_{\R^d}p\left(t/2,x,z\right)p\left(t/2,z,y\right)dz\\
&=\left(\int_{\{z:|z-y|> |x-y|/2\}}+\int_{\{z:|z-y|\le |x-y|/2\}}\right)
p\left(t/2,x,z\right)p\left(t/2,z,y\right)dz\\
&=:I_1+I_2
\end{align*}
According to \eqref{l4-3-4}, we have
\begin{align*}
I_1&\le \sup_{z:|z-y|>|x-y|/2}p\left(t/2,z,y\right)\cdot \int_{\R^d}p\left(t/2,x,z\right)\,dz\\
&\le c_{21}e^{-c_{22}t}\exp\left(-\frac{c_{23}(1+|y|)^2}{t}\right),
\end{align*}
where we have used the fact
\begin{align*}
\int_{\R^d}p\left(t/2,x,z\right)dz=T_{t/2}^V1(x)\le c_{24}e^{-c_{25}t}.
\end{align*} thanks to \eqref{l2-3-1}.
On the other hand, following the same argument as above and using \eqref{l4-3-3}, we deduce
\begin{align*}
I_2\le  c_{26}e^{-c_{27}t}\exp\left(-\frac{c_{27}(1+|y|)^2}{t}\right).
\end{align*}
Putting both estimates above for $I_1$ and $I_2$ together, we can get the desired assertion \eqref{l4-3-1} when $C_0t_0(|x|)< t\le C_0't_0(|y|)$.

Next, we consider the case that $t>C_0't_0(|y|)$.
Then, by \eqref{r2-1-}, we obtain
\begin{equation*}
T_{t/2}^V 1(y)\le c_{28}e^{-c_{29}t}\exp\left(-c_{29}(1+|y|)\sqrt{g(|y|)}\right).
\end{equation*}
Therefore,
\begin{equation}\label{l4-3-5}
\begin{split}
p(t,x,y)&=\int_{\R^d}p\left(t/2,x,z\right)p\left(t/2,z,y\right)dz\le \sup_{z\in \R^d}p\left(t/2,x,z\right)\cdot T_{t/2}^V 1(y)\\
&\le c_{30}t^{-d/2}e^{-c_{29}t}\exp\left(-c_{29}(1+|y|)\sqrt{g(|y|)}\right)\\
&\le c_{31}e^{-c_{29}t}\exp\left(-c_{29}(1+|y|)\sqrt{g(|y|)}\right),
\end{split}
\end{equation}
where we used the facts that $t\ge C_0't_0(|y|)\ge c_{32}$ and
\begin{align*}
\sup_{z\in \R^d}p\left(t/2,x,z\right)\le
\sup_{z\in \R^d}q\left(t/2,x,z\right)\le c_{33}t^{-d/2}\le c_{34}.
\end{align*}
In particular, \eqref{l4-3-5} implies that \eqref{l4-3-1} holds when $t>C_0't_0(|y|)$, due to the fact that
\begin{align*}
\frac{(1+|y|)^2}{t}\le c_{35}(1+|y|)\sqrt{g(|y|)},\quad t>C_0't_0(|y|)\hbox{ and } |y|\ge 2.
\end{align*}

{\bf Case 2: $|y|\ge 2$ and $|y|\le 2|x|$.} By the definition of $t_0(\cdot)$ and the fact that $|x|\le |y|\le 2|x|$, there exist positive constants $c_{36}, c_{37}$ such that
$
c_{36}t_0(|x|)\le t_0(|y|)\le c_{37}t_0(|x|),
$
which implies that $t\ge C_0t_0(|x|)\ge C_0c_{37}^{-1}t_0(|y|)\ge c_{38}$.
Therefore, following the same arguments for \eqref{l4-3-5}, we can prove that for all $t>C_0t_0(|x|)$,
\begin{align*}
p(t,x,y)\le c_{39}e^{-c_{40}t}\exp\left(-c_{40}(1+|y|)\sqrt{g(|y|)}\right)
\end{align*}

Meanwhile, the fact $t>C_0t_0(|x|)\ge C_0c_{37}^{-1}t_0(|y|)$ implies that
\begin{align*}
 \exp\left(-c_{40}(1+|y|)\sqrt{g(|y|)}\right)\le c_{41}\exp\left(\frac{-c_{42}(1+|y|)^2}{t}\right).
\end{align*}

Putting both estimates above together,
we can prove \eqref{l4-3-1} when  $|y|\ge 2$ and $|y|\le 2|x|$.

{\bf Case 3:  $|y|\le 2$.} If $|x|\le |y|\le 2$, then one can find a positive constant $c_{43}$ such that
$t \ge C_0t_0(|x|)\ge c_{43}$.
Hence,
\begin{align*}
p(t,x,y)&=\int_{\R^d}p\left(t/2,x,z\right)p\left(t/2,z,y\right)dz\le \sup_{z\in \R^d}p\left(t/2,x,z\right)\cdot T_{t/2}^V 1(y)\le c_{44}t^{-d/2}e^{-c_{45}t}\\
&\le  c_{46}e^{-c_{47}t}\min\left\{\exp\left(-c_{47}(1+|y|)\sqrt{g(|y|)}\right),\exp\left(-\frac{c_{47}(1+|y|)^2}{t}\right)\right\},
\end{align*}
where the last inequality follows from the fact for every $t\ge C_0t_0(|x|)\ge c_{43}$ and $|y|\le 2$,
\begin{align*}
\min\left\{\exp\left(-c_{47}(1+|y|)\sqrt{g(|y|)}\right),\exp\left(-\frac{c_{47}(1+|y|)^2}{t}\right)\right\}\ge c_{48}.
\end{align*}
By now we have finished the proof. \end{proof}

\begin{lemma}\label{l4-4}
Assume that $t_0(\cdot)$ is almost increasing. For any $C_0>0$, there are positive constants $C_{7}$ and $C_{8}$ such that
for all $x,y\in \R^d$ and $t>C_0t_0(|x|)$,
\begin{equation}\label{l4-4-1}
p(t,x,y)\ge C_{7}e^{-C_{8}t}\min\left\{\exp\left(-\frac{C_{8}(1+|y|)^2}{t}\right),
\exp\left(-C_{8}(1+|y|)\sqrt{g(|y|)}\right)\right\}.
\end{equation}
\end{lemma}
\begin{proof} The proof is split into three cases.

{\bf Case 1:  $|x|\ge 4$.} By \eqref{l4-2-3} (noting that in the proof of \eqref{l4-2-2a} we do not use the almost decreasing property of $t_0(\cdot)$), we know that for all $|u|\ge 4$, $z\in B(0,2)$ and  $t>C_0t_0(|u|)$,
\begin{equation}\label{l4-4-2}
p(t,u,z)\ge c_1e^{-c_2t}\exp\left(-c_2|u|\sqrt{g(|u|)}\right).
\end{equation}
 On the other hand, for
$|u|\ge 4$, $z\in B(0,2)$ and $C_0t_0(|x|)\le t\le C_0't_0(|u|)$, one can follow the proof of \eqref{l4-2-3} to see that the fourth step in \eqref{l4-2-2a} still holds. Thus,
\begin{align*}
p(t,u,z)
&\ge c_3
\exp\left(-c_4\left(t+\frac{1}{t}\right)\right)
\int_0^{t/2}\exp\left(-c_4\left(sg(|u|)+\frac{|u|^2}{s}+\frac{s}{|u|^2}\right)\right)s^{-(d/2+1)}\,ds\\
&\ge c_5e^{-c_6t}\int_{{t}/{4}}^{t/2}\exp\left(-c_6\left(sg(|u|)+\frac{|u|^2}{s}+\frac{s}{|u|^2}\right)\right)s^{-(d/2+1)}\,ds\\
&\ge c_7t^{-d/2}e^{-c_6t}\exp\left(-c_8\left(tg(|u|)+\frac{|u|^2}{t}+\frac{t}{|u|^2}\right)\right)\ge c_9e^{-c_{10}t}\exp\left(-\frac{c_{10}|u|^2}{t}\right).
\end{align*}
Here in the second inequality we have used the property
\begin{align*}
\exp\left(-c_4\left(t+\frac{1}{t}\right)\right)
\ge c_{11}e^{-c_{12}t},\quad  t>C_0t_0(|x|)\ge c_{13},
\end{align*}
and the last inequality follows from the fact that for all $u\in \R^d$ with $|u|\ge4$,
\begin{align*}
\frac{|u|^2}{t}\ge c_{14}\left( 
tg(|u|)+\frac{t}{|u|^2}\right)\ge c_{15},\quad c_{13}\le t\le C_0t_0(|u|)\le c_{16}|u|^2.
\end{align*}
Combining all the estimates above for $p(t,u,z)$ together yields that for every
$z\in B(0,2)$ (by considering the different cases that $t\le C_0't_0(|y|)$ and $t>C_0't_0(|y|)$)
\begin{equation}\label{l4-4-3}
p(t,y,z)\ge c_{17}e^{-c_{18}t}\min\left\{\exp\left(-c_{18}|y|\sqrt{g(|y|)}\right), \exp\left(-\frac{c_{18}|y|^2}{t}\right)\right\}.
\end{equation}
Hence, by \eqref{l4-4-2} and \eqref{l4-4-3}, we obtain that for $t>C_0t_0(|x|)$,
\begin{align*}
  p(t,x,y)
&\ge \int_{B(0,2)}p\left(t/2,x,z\right)p\left(t/2,z,y\right)dz\\
&\ge c_{19}\inf_{z\in B(0,2)}p\left(t/2,x,z\right)\cdot \inf_{z\in B(0,2)}p\left(t/2,z,y\right)\\
&\ge c_{20}e^{-c_{21}t}\exp\left(-c_{21}|x|\sqrt{g(|x|)}\right) e^{-c_{21}t}\min\left\{ \exp\left(-c_{21}|y|\sqrt{g(|y|)}\right), \exp\left(-\frac{c_{21}|y|^2}{t}\right)\right\}\\
&\ge c_{20}e^{-2c_{21}t}\min\left\{\exp\left(-2c_{21}(1+|y|)\sqrt{g(|y|)}\right), \exp\left(-\frac{2c_{21}(1+|y|)^2}{t}\right)\right\}.
\end{align*}
Thus, we have shown \eqref{l4-4-1} for $|y|\ge |x|\ge 4$.

{\bf Case 2:  $|x|\le 4$ and $|y|\ge 4$.}  When $|x|\le 4$ and $|y|\ge 4$, we can use \eqref{l4-2-2} instead of \eqref{l4-4-2}, and then follow the arguments in
{\bf Case 1}  (using \eqref{l4-2-2} and \eqref{l4-4-3}) to prove \eqref{l4-4-1}.

{\bf Case 3:  $|x|\le 4$ and $|y|\le 4$.} According to \eqref{l4-2-2} and the fact $t>C_0t_0(|x|)\ge c_{22}$, we can verify \eqref{l4-4-1} by the same procedure
as that of  {\bf Case 3} in the proof of Lemma \ref{l4-2}.
\end{proof}

Putting  all Lemmas \ref{l4-1}--\ref{l4-4} together, we will obtain the two-sided estimates of
$p(t,x,y)$ for all $t>C_0t_0(|x|)$.
\begin{proposition}\label{p4-1}
Suppose that $t_0(\cdot)$ is either almost increasing or almost decreasing.
For any $C_0>0$, there exist positive constants $C_{i}$, $9\le i \le 12$, such that
for every $x,y\in \R^d$ with $|x|\le |y|$ and  $t>C_0t_0(|x|)$,
\begin{equation}\label{p3-1-1}
\begin{split}
&p(t,x,y)\ge C_{9}e^{-C_{10}t}\min\left\{
\exp\left(-C_{10}(1+|y|)\sqrt{g(|y|)}\right), \exp\left(-\frac{C_{10}(1+|y|)^2}{t}\right)\right\},\\
&p(t,x,y)\le C_{11}e^{-C_{12}t}\min\left\{
\exp\left(-C_{12}(1+|y|)\sqrt{g(|y|)}\right), \exp\left(-\frac{C_{12}(1+|y|)^2}{t}\right)\right\}.
\end{split}
\end{equation}
\end{proposition}

\section{Proofs}\label{sec5}

\begin{proof}[Proof of Theorem $\ref{thm1}$]
Theorem \ref{thm1} immediately follows from Propositions \ref{p3-1} and \ref{p4-1}.
\end{proof}

\begin{proof}[Proof of Example $\ref{ex1-1}$] Let $V(x)=|x|^\alpha$ with $\alpha>0$.
Taking $g(r)=(1+r)^\alpha$ and applying Theorem \ref{thm1}, we find that for every $x,y\in \R^d$ with $|x|\le |y|$,
\begin{equation}\label{ex1-1-4}
\begin{split}
p(t,x,y)\asymp
\begin{cases}
t^{-d/2}e^{-\frac{|x-y|^2}{t}}e^{-t(1+|x|)^\alpha-|x-y|(1+|y|)^{{\alpha}/{2}}},
&\quad t\le (1+|x|)^{1-{\alpha}/{2}},\\
e^{-t}e^{ -(1+|y|)^{1+{\alpha}/{2}} -\frac{(1+|y|)^2}{t}},&\quad t>
(1+|x|)^{1-{\alpha}/{2}}.
\end{cases}
\end{split}
\end{equation}

Next, we first consider $\alpha\in (2,\infty)$. In this case, due to $|x|\le |y|$,
$$\max\left\{(1+|x|)^{1-{\alpha}/{2}}, (1+|y|)^{1-{\alpha}/{2}}\right\}=(1+|x|)^{1-{\alpha}/{2}}.$$
When $x,y\in \R^d$ with $|x-y|\le |y|/2$ or $|y|\le 2$ (which implies that $|y|/2\le |x|\le |y|$ 
in particular
when $|x-y|\le |y|/2$), for every $t>0$ it holds
$$|x-y|(1+|y|)^{{\alpha}/{2}}\le \frac{|x-y|^2}{2t}+\frac{t(1+|y|)^{{\alpha}/{2}}}{2}\le \frac{|x-y|^2}{2t}+\frac{t(
3+2|x|)^{{\alpha}/{2}}}{2}$$ and $$(1+|x|)^\alpha\asymp  (1+|y|)^\alpha.$$
Hence, when $0\le t\le (1+|x|)^{1-\alpha/2}$, we can verify \eqref{ex1-1-2} by the first estimate in \eqref{ex1-1-4}.
Meanwhile, when $t>(1+|x|)^{1-\alpha/2}$, \eqref{ex1-1-2}
is a consequence of the second estimate in \eqref{ex1-1-4}, thanks to the fact that $(1+|y|)^{1+\alpha/2}\ge (1+|y|)^2/t$ for all $t>(1+|x|)^{1-\alpha/2}\ge (1+|y|)^{1-\alpha/2}$.

When
$|x-y|>|y|/2$ and $|y|>2$ (which implies that $|y|/2< |x-y|\le 2|y|$),
$$
 |x-y|(1+|y|)^{{\alpha}/{2}}\asymp (1+|y|)^{1+{\alpha}/{2}}$$ and,  for $0<t\le (1+|x|)^{1-{\alpha}/{2}}$,
$$t(1+|x|)^\alpha \le \frac{(1+|x|)^2}{t}
\le \frac{(1+|y|)^2}{t}
\le \frac{c_1|x-y|^2}{t}.$$ Hence, we can see that \eqref{ex1-1-3} holds true, thanks to \eqref{ex1-1-4} again.

Now, we consider $\alpha\in (0,2]$. Then, due to $|x|\le |y|$,
$$\max\left\{(1+|x|)^{1-{\alpha}/{2}}, (1+|y|)^{1-{\alpha}/{2}}\right\}=(1+|y|)^{1-{\alpha}/{2}}.$$
Suppose that $x,y\in \R^d$ with
$|x-y|\le |y|/2$ or $|y|\le 2$.
Then, following the argument above, we can see from \eqref{ex1-1-4} that \eqref{ex1-1-2} holds when $t\le (1+|x|)^{1-{\alpha}/{2}}$. Furthermore, when
 $(1+|x|)^{1-{\alpha}/{2}}<t\le (1+|y|)^{1-{\alpha}/{2}}$,
\begin{equation}\label{ex1-1-1} \frac{(1+|y|)^2}{t}\asymp t(1+|y|)^\alpha\asymp (1+|y|)^{1+{\alpha}/{2}},\quad c_2t^{-d/2}e^{-c_3t(1+|y|)^\alpha}\le e^{-c_4t}\le c_5t^{-d/2}.\end{equation}
Combining all the properties above with \eqref{ex1-1-4},  we can prove that \eqref{ex1-1-2} holds
when $|x-y|\le |y|/2$ or $|y|\le 2$.
Similarly, we can apply the arguments above to see that  \eqref{ex1-1-3} holds
when $|x-y|>|y|/2$ and $|y|>2$.
\end{proof}

\begin{proof}[Proof of Proposition $\ref{t1-3}$]
By the symmetry property $G(x,y)=G(y,x)$, we will assume that $|x|\le |y|$.

{\bf Case 1}:
$|x-y|\le |y|/2$ or $|y|\le 2$.
Under the condition $|x-y|\le |y|/2$ or $|y|\le 2$, we can see from the proof of Example \ref{ex1-1} (by noting that here $t_0(|x|)\asymp t_0(|y|)$
and $|x|\asymp |y|$) and obtain
\begin{equation}\label{t1-3-2}
\begin{split}
p(t,x,y)\asymp
\begin{cases}
t^{-d/2}e^{-\frac{|x-y|^2}{t}}e^{-tg(|y|)},\ \ & t\le t_0(|y|),\\
e^{-t}e^{-(1+|y|)\sqrt{g(|y|)}},\ \ \ & t>t_0(|y|).
\end{cases}
\end{split}
\end{equation}

{\bf Subcase 1.1}: $|x-y|^2\le t_0(|y|)/2$.
Now we set
\begin{align*}
G(x,y)&=\int_0^\infty p(t,x,y)\, dt=
\left(\int_0^{|x-y|^2}+\int_{|x-y|^2}^{t_0(|y|)}+\int_{t_0(|y|)}^\infty\right)p(t,x,y)\,dt\\
&=:I_1+I_2+I_3.
\end{align*}
According to \eqref{t1-3-2},
\begin{align*}
I_1&\asymp \int_0^{|x-y|^2} t^{-d/2}e^{-\frac{|x-y|^2}{t}}e^{-tg(|y|)}\,dt\\
&\asymp |x-y|^{-(d-2)}\int_1^\infty s^{{d}/{2}-2}\exp\left(-\left(s+\frac{g(|y|)|x-y|^2}{s}\right)\right)\,ds\\
&\asymp |x-y|^{-(d-2)}\exp\left(-|x-y|\sqrt{g(|y|)}\right).
\end{align*}
Here the second step follows from the change of variable $s=\frac{|x-y|^2}{t}$, and in the last step we have used
the fact that
\begin{align}\label{t1-3-3}
\int_1^\infty s^{{d}/{2}-2}e^{-\left(s+\frac{R^2}{s}\right)}\,ds\asymp e^{-R},\quad R>0.
\end{align}

Again by \eqref{t1-3-2},
\begin{align*}
I_2&\asymp \int_{|x-y|^2}^{t_0(|y|)}t^{-d/2}e^{-tg(|y|)}\,dt\asymp |x-y|^{-(d-2)}\int_{1}^{\frac{t_0(|y|)}{|x-y|^2}}s^{-d/2}\exp\left(-s |x-y|^2 g(|y|)\right)\,ds,
\end{align*}
where we have used the change of variable $s=\frac{t}{|x-y|^2}$ in the last step. Note that $\frac{t_0(|y|)}{|x-y|^2}\ge 2$ and $\frac{t_0(|y|)}{|x-y|^2}\ge \frac{1}{|x-y|^2g(|y|)}$. If $d>2$ or
$|x-y|^2g(|y|)>1$, then
\begin{align*}
 &\int_{1}^{\frac{t_0(|y|)}{|x-y|^2}}s^{-d/2}\exp\left(- s |x-y|^2 g(|y|)\right)\,ds \\
&\le \exp\left(-c_1|x-y|^2g(|y|)\right)\int_{1}^{\infty}s^{-d/2}\exp\left(- c_2 s |x-y|^2 g(|y|)\right)\,ds \\
&\le c_3\exp\left(-c_4|x-y|^2g(|y|)\right)
\end{align*} and
\begin{align*}
  \int_{1}^{\frac{t_0(|y|)}{|x-y|^2}}s^{-d/2}\exp\left(- s |x-y|^2 g(|y|)\right)\,ds
&\ge \int_{1}^{2}s^{-d/2}\exp\left(- s |x-y|^2 g(|y|)\right)\,ds\\
&\ge c_5\exp\left(-2|x-y|^2g(|y|)\right).
\end{align*}
If $d=2$ and $|x-y|^2g(|y|)\le 1$, then
\begin{align*}
&\int_{1}^{\frac{t_0(|y|)}{|x-y|^2}}s^{-1}\exp\left(- s |x-y|^2 g(|y|)\right)\,ds\\
&\le \int_1^{\max\left\{\frac{1}{|x-y|^2g(|y|)},2\right\}}s^{-1}e^{- s |x-y|^2 g(|y|)}\,ds+
\int_{\max\left\{\frac{1}{|x-y|^2g(|y|)},2\right\}}^\infty s^{-1}e^{- s |x-y|^2 g(|y|)}\,ds\\
&\le c_6\int_1^{\max\left\{\frac{1}{|x-y|^2g(|y|)},2\right\}}s^{-1}\,ds+\int_1^\infty s^{-1}e^{-s}\,ds\\
&\le c_7\max\left\{\log\left(\frac{1}{|x-y|^2g(|y|)}\right),1\right\}
\end{align*} and
\begin{align*}
 \int_{1}^{\frac{t_0(|y|)}{|x-y|^2}}s^{-1}\exp\left(- s |x-y|^2 g(|y|)\right)\,ds
&\ge \int_1^{\max\left\{\frac{1}{|x-y|^2g(|y|)},2\right\}}s^{-1}e^{- s |x-y|^2 g(|y|)}\,ds\\
&\ge c_8\max\left\{\log\left(\frac{1}{|x-y|^2g(|y|)}\right),1\right\}.
\end{align*}
Applying the same arguments as above, we can obtain that when $d=1$ and $|x-y|^2g(|y|)\le 1$,
\begin{align*}
\int_{1}^{\frac{t_0(|y|)}{|x-y|^2}}s^{-1/2}\exp\left(- s |x-y|^2 g(|y|)\right)\,ds\asymp
\frac{1}{|x-y|\sqrt{g(|y|)}}.
\end{align*}

Therefore, putting all the estimates above together yields that
\begin{align*}
I_2\asymp & |x-y|^{-(d-2)}\exp\left(-|x-y|^2g(|y|)\right)\cdot
\begin{cases}
1, & d\ge 3,\\
\max\left\{\log\left(\frac{1}{|x-y| \sqrt{g(|y|)}}\right),1\right\},\ &d=2,\\
\frac{1}{|x-y|\sqrt{g(|y|)}},\ &d=1
\end{cases}\\
= :&|x-y|^{-(d-2)} \Pi(x,y).
\end{align*}

Furthermore, using \eqref{t1-3-2}, we can obtain
\begin{align}\label{t1-3-4}
I_3\asymp \int_{t_0(|y|)}^\infty e^{-t}\exp\left(-(1+|y|)\sqrt{g(|y|}\right)\,dt
\asymp \exp\left(-(1+|y|)\sqrt{g(|y|}\right).
\end{align}

Combining with all the estimates for $I_1,I_2$ and $ I_3$, we can get that
\begin{align*}G(x,y)\asymp & |x-y|^{-(d-2)}\left[\exp\left(-|x-y|\sqrt{g(|y|)}\right)+\Pi(x,y)\right]+\exp\left(-(1+|y|)\sqrt{g(|y|}\right)\\
\asymp & |x-y|^{-(d-2)} \Gamma(x,y)+\exp\left(-(1+|y|)\sqrt{g(|y|}\right),\end{align*} where $\Gamma(x,y)$ is defined by \eqref{t1-3-1a}.

Note that, when $d \ge 2$ and $|x-y|\ge 1$,
\begin{align*}
|x-y|^{-(d-2)} \Gamma(x,y)&=g(|y|)^{\frac{d-2}{2}}\left(|x-y|\sqrt{g(|y|)}\right)^{-(d-2)}\Gamma(x,y)\\
&\ge c_9\exp\left(-c_{10}|x-y|\sqrt{g(|y|)}\right)\\
&\ge c_9\exp\left(-2c_{10}(1+|y|)\sqrt{g(|y|)}\right);
\end{align*} 
when $d\ge 2$ and $|x-y|\le 1$,
$$|x-y|^{-(d-2)} \Gamma(x,y) \ge\Gamma(x,y)\ge  c_9\exp\left(-c_{10}(1+|y|)\sqrt{g(|y|}\right);$$
when $d=1$,
$$|x-y|^{-(d-2)} \Gamma(x,y)=
(1+|y|)\cdot\frac{\exp(-|x-y|\sqrt{g(|y|)})}{(1+|y|)\sqrt{g(|y|)}}
\ge c_9\exp\left(-c_{10}(1+|y|)\sqrt{g(|y|}\right).$$
This along with the estimate yields the desired conclusion
\eqref{t1-3-1} for every $x,y\in \R^d$ with
$|x-y|^2\le t_0(|y|)/2$.

{\bf Subcase 1.2}: $|x-y|^2> t_0(|y|)/2$. In this case, we define
$$
G(x,y) =\int_0^{t_0(|y|)/2}p(t,x,y)\,dt+\int_{t_0(|y|)/2}^\infty p(t,x,y)\,dt=:J_1+J_2.
$$

By \eqref{t1-3-2}, we derive
\begin{align*}
J_1&\asymp \int_0^{t_0(|y|)/2}t^{-d/2}e^{-\frac{|x-y|^2}{t}}e^{-tg(|y|)}\,dt\\
&\asymp |x-y|^{-(d-2)}\int_{\frac{t_0(|y|)}{2|x-y|^2}}^\infty s^{{d}/{2}-2}\exp\left(-\left(s+\frac{g(|y|)|x-y|^2}{s}\right)\right)\,ds\\
&\asymp |x-y|^{-(d-2)}\exp\left(-|x-y|\sqrt{g(|y|)}\right),
\end{align*}
where in the second step we have used the change of variable $s=\frac{|x-y|^2}{t}$, and the last step is due to
\eqref{t1-3-3} and the fact that
\begin{align*}
\int_0^1 s^{{d}/{2}-2}\exp\left(-c_{11}\left(s+\frac{g(|y|)|x-y|^2}{s}\right)\right)\,ds
\le c_{12}\exp\left(-c_{13}|x-y|\sqrt{g(|y|)}\right).
\end{align*}

According to \eqref{t1-3-4}, we have
\begin{align*}
J_2\asymp \exp\left(-(1+|y|)\sqrt{g(|y|)}\right).
\end{align*}

Let $\Gamma(x,y)$ be defined by \eqref{t1-3-1a}. Note that when $|x-y|^2\ge t_0(|y|)/2$, $|x-y| \sqrt{g(|y|)}\ge c_{14}$, and so
$$\Gamma(x,y)\asymp \exp\left(-|x-y|\sqrt{g(|y|)}\right).
$$

Hence, according to all the estimates above and the arguments in the end of {\bf Subcase 1.1}, we obtain that for all $x,y\in \R^d$ with
$|x-y|^2> t_0(|y|)/2$,
\begin{align*}
G(x,y)&\asymp |x-y|^{-(d-2)}\exp\left(-|x-y|\sqrt{g(|y|)}\right)
\asymp |x-y|^{-(d-2)}\Gamma(x,y),
\end{align*}
where  
in the last step we used the facts that $|x-y| \sqrt{g(|y|)}\ge c_{15}$ due to $|x-y|^2\ge t_0(|y|)/2\ge c_{16}|x-y|/\sqrt{g(|y|)}$, and that,
for $d=1$,
\begin{align*}
&\frac{c_{17}}{\sqrt{g(|y|)}}\exp\left(-c_{18}|x-y|\sqrt{g(|y|)}\right)\\
&\le |x-y|\exp\left(-|x-y|\sqrt{g(|y|)}\right) 
=\frac{|x-y|\sqrt{g(|y|)}}{\sqrt{g(|y|)}}\exp\left(-|x-y|\sqrt{g(|y|)}\right)\\
&\le \frac{c_{19}}{\sqrt{g(|y|)}}\exp\left(-c_{20}|x-y|\sqrt{g(|y|)}\right),\end{align*} thanks to
$|x-y| \sqrt{g(|y|)}\ge c_{15}$ again.

{\bf Case 2}:
$|x-y|>|y|/2$ and $|y|>2$. Then we can follow the proof of Example \ref{ex1-1} and obtain
\begin{equation}\label{t1-3-5}
\begin{split}
p(t,x,y)\asymp
\begin{cases}
t^{-d/2}e^{-\frac{|x-y|^2}{t}}e^{-(1+|y|)\sqrt{g(|y|)}},\ \ & t\le \max\{t_0(|x|),t_0(|y|)\},\\
e^{-t}e^{-(1+|y|)\sqrt{g(|y|)}},\ \ \ & t>\max\{t_0(|x|),t_0(|y|)\}.
\end{cases}
\end{split}
\end{equation}

Set
$$
G(x,y) =\int_0^{\max\{t_0(|x|),t_0(|y|)\}}p(t,x,y)\,dt+\int_{\max\{t_0(|x|),t_0(|y|)\}}^\infty p(t,x,y)\,dt=:K_1+K_2.
$$

By \eqref{t1-3-5}, it holds that
\begin{align*}
K_1&\asymp \int_{0}^{\max\{t_0(|x|),t_0(|y|)\}}t^{-d/2}e^{-\frac{|x-y|^2}{t}}e^{-(1+|y|)\sqrt{g(|y|)}}\,dt\\
&\asymp |x-y|^{-(d-2)}e^{-(1+|y|)\sqrt{g(|y|)}}\int_{\frac{|x-y|^2}{\max\{t_0(|x|),t_0(|y|)\}}}^\infty
s^{{d}/{2}-2}e^{-s}\,ds\\
&\asymp \exp\left(-(1+|y|)\sqrt{g(|y|)}\right)\asymp |x-y|^{-(d-2)}\exp\left(-|x-y|\sqrt{g(|y|)}\right).
\end{align*}
where in the second step we have used the change of variable $s=\frac{|x-y|^2}{t}$, and the third step follows from
the fact that (by noting that $\frac{|x-y|^2}{\max\{t_0(|x|),t_0(|y|)\}}\asymp \frac{(1+|y|)^2}{\max\{t_0(|x|),t_0(|y|)\}}\ge c_{21}$)
\begin{align*}
c_{22}\exp\left(-c_{23}(1+|y|)\sqrt{g(|y|)}\right)\le \int_{\frac{|x-y|^2}{\max\{t_0(|x|),t_0(|y|)\}}}^\infty
s^{{d}/{2}-2}e^{-s}\,ds\le c_{24}
\end{align*}

Still by \eqref{t1-3-4} and \eqref{t1-3-5}, we obtain
\begin{align*}
K_2&\asymp e^{-(1+|y|)\sqrt{g(|y|)}}\int_{\max\{t_0(|x|),t_0(|y|)\}}^\infty e^{-s}\,ds\\
&\asymp |x-y|^{-(d-2)}\exp\left(-(1+|y|)\sqrt{g(|y|)}\right)\asymp |x-y|^{-(d-2)}\exp\left(-|x-y|\sqrt{g(|y|)}\right),
\end{align*} where we used the fact that $|x-y|\asymp |y|$.

Putting both estimates for $K_1$, $K_2$ together yields that
\begin{align*}
G(x,y)\asymp |x-y|^{-(d-2)}\exp\left(-|x-y|\sqrt{g(|y|)}\right)\asymp |x-y|^{-(d-2)}\Gamma(x,y),
\end{align*}
where the last step is due to the fact that
\begin{align*}
\Gamma(x,y)\asymp \exp\left(-|x-y|\sqrt{g(|y|)}\right),\quad  x,y\in \R^d\ {\rm with}\ |x-y|\asymp (1+|y|)\ge\frac{c_{25}}{\sqrt{g(|y|)}}.
\end{align*}

By now, we have finished the proof of the desired conclusion \eqref{t1-3-1}.
\end{proof}

\bigskip

\noindent{\bf Acknowledgement.} The research of Xin Chen is supported by
the National Natural Science Foundation of China (No.\ 12122111).
The research of Jian Wang is supported by the National Key R\&D Program of China (2022YFA1000033), and 
the National Natural Science Foundation of China (Nos. 11831014, 12071076 and 12225104).

\end{document}